\documentclass[10pt]{amsart}

\usepackage[utf8x]{inputenc}

\usepackage{amssymb}
\usepackage{amsmath,amsfonts,amsthm,amstext, amsthm}
\usepackage[all,cmtip]{xy}
\usepackage[dvips]{graphicx}
\usepackage{graphics}
\usepackage{calrsfs}
\usepackage{hyperref}
\usepackage{xcolor}
  \usepackage{tikz}
\usepackage{tikz-cd}
\usetikzlibrary{positioning}
\usepackage{pgf, tikz}
\usetikzlibrary{arrows, automata}
\usepackage{mathtools}

\newtheorem{theorem}{Theorem}[section]
\newtheorem{lemma}[theorem]{Lemma}
\newtheorem{prop}[theorem]{Proposition}

\newtheorem{crlr}[theorem]{Corollary}
\newtheorem{corollary}[theorem]{Corollary}

\newtheorem{teo}{Theorem}

\theoremstyle{definition}
\newtheorem{definition}[theorem]{Definition}

\newtheorem{remark}[theorem]{Remark}

\newtheorem*{question*}{Question}

\newcommand{\cd}{{\rm cd}}

\newcommand{\rk}{{\rm rk}}

\newcommand{\lk}{{\rm lk}}

\newcommand{\Out}{{\rm Out}}
\newcommand{\Inn}{{\rm Inn}}
\newcommand{\Aut}{{\rm Aut}}
\newcommand{\GL}{{\rm GL}}

\newcommand{\Ext}{{\rm Ext}}

\newcommand{\st}{{\rm st}}

\newcommand{\Q}{\mathbb Q}
\newcommand{\Z}{\mathbb Z}
\newcommand{\N}{\mathbb N}
\newcommand{\R}{\mathbb R}

\newcommand{\X}{\mathfrak{X}}

\newcommand{\FP}{\operatorname{FP}}

\newcommand{\PAut}{\operatorname{PAut}}
\newcommand{\POut}{\operatorname {POut}}
\newcommand{\Der}{\operatorname {Der}}

\newcommand{\g}{\mathfrak{g}}
\newcommand{\h}{\mathfrak{h}}

\newcommand{\al}{\mathfrak{a}}

\newcommand{\im}{\operatorname {Im}}

\newcommand{\id}{\mathrm{id}}
\newcommand{\gr}{\mathrm{gr}_\bullet}

\newcommand{\ad}{\operatorname{ad}}

\theoremstyle{definition}

\numberwithin{equation}{section}
\begin{document}

\title{Pure symmetric automorphisms, extensions of RAAGs, and Koszulness}

\author[C.~Mart\'inez-P\'erez]{Conchita Mart\'inez-P\'erez}
\address{Departamento de Matem\' aticas\\
        Universidad de Zaragoza\\
         50009 Zaragoza, Spain}
\email{conmar@unizar.es}

\author[L.~Mendon\c{c}a]{Luis Mendon\c{c}a}
\address{Instituto de Ci\^encias Exatas\\
        Universidade Federal de Minas Gerais\\
        31270-901 Belo Horizonte, Brazil}
\email{luismendonca@mat.ufmg.br}


\begin{abstract}
We characterize in terms of a combinatorial condition on the graph $\Gamma$ when the group $\PAut(A_\Gamma)$ of pure symmetric automorphisms of the RAAG $A_\Gamma$ and its outer version $\POut(A_\Gamma)$  have a descending central Lie algebra which is Koszul. To do that, we prove that our combinatorial condition implies that these groups are iterated extensions of RAAGs; in particular, they are poly-free. On the other hand, we show that $\PAut(F_n)$ is not poly-finitely generated free for $n \geq 4$. We also show that groups in a certain class containing $\PAut(A_\Gamma)$ are 1-formal.
\end{abstract}

\maketitle

\section{Introduction}
The group of pure-symmetric automorphisms of a right-angled Artin group $A_\Gamma$ is the subgroup $\PAut(A_\Gamma) \subseteq \Aut(A_\Gamma)$ consisting of the automorphisms that send each standard generator $v \in \Gamma$ to a conjugate $v^g$. Among these are the \emph{partial conjugations}, that is, the automorphisms $\phi = c_L^v$ with
\[ w^\phi =
\begin{cases}
w^v & \text{if $w \in L$}, \\
w & \text{if $w \notin L$},
\end{cases}
\]
where $L$ is a connected component of $\Gamma \smallsetminus \st_\Gamma(v)$ and $\st_\Gamma(v)$ is the star of $v$ in $\Gamma$. In fact Laurence proves in \cite{Laurence} that $\PAut(A_\Gamma)$ is generated by all partial conjugations.

The following terms are useful in the analysis of relations holding among partial conjugations. For a pair of non-adjacent vertices $v,w \in \Gamma$, the connected components $D_w \subseteq \Gamma \smallsetminus \st(v)$ and
$D_v \subseteq \Gamma \smallsetminus \st(w)$ that contain $w$ and $v$, respectively, will be called \emph{dominant} components. A connected component $L\subseteq \Gamma\smallsetminus\st(w)$ such that $L\subseteq D_w$ or the other way around will be called \emph{subordinate}, and a subgraph $L \subseteq \Gamma$ that is a connected component of both $\Gamma \smallsetminus \st(v)$ and $\Gamma \smallsetminus \st(w)$ will be called a \emph{shared} component. Notice that all these notions are relative to the pair $(v,w)$, and it can be shown that any connected component of $\Gamma \smallsetminus \st(v)$ or $\Gamma \smallsetminus \st(w)$ falls into one of these three possibilities. Following \cite{CRSV}, we say that $v$ and $w$ form a SIL-pair (for \emph{separating intersection of links}) if $v$ and $w$ have at least one shared component.

A presentation of $\PAut(A_\Gamma)$ is as follows. It is the group generated by all possible partial conjugations $c_L^v$, and the following defining relations:
\begin{enumerate}
 \item $[c_L^v, c_K^w]=1$ if $v \in \st(w)$, or one of the components $L$ or $K$ is subordinate, or $L$ and $K$ are distinct shared components,
 \item $[c_L^v, c_L^w c_K^w]=1$ when $L$ is shared and $K$ is dominant.
 \end{enumerate}
For more details, see \cite[Chapter~3]{Toinet}.

The special case of $\PAut(F_n)$, where $F_n$ is free of rank $n$, is also known as the McCool group, the group of welded pure braids or the group of basis-conjugating automorphisms. If $\{x_1, \ldots, x_n\}$ is a free basis of $F_n$, we denote for simplicity by $c_{ij}$ the partial conjugation $c_{\{x_i\}}^{x_j}$. So $\PAut(F_n)$ is generated by all $c_{ij}$, with $1 \leq i \neq j \leq n$ and has defining relators
\[ [c_{ij},c_{kj}], \  \ [c_{ij},c_{kl}] \ \ \text{ and } \ \ [c_{ij}, c_{jk}c_{ik}]\]
for $i,j,k,l$ all distinct.

Some recent results give evidence that $\PAut(A_{\Gamma})$ is in some sense close to a RAAG if the defining graph $\Gamma$ is not too complicated. For instance, Koban and Piggott showed that $\PAut(A_\Gamma)$ is isomorphic to a RAAG if and only if $\Gamma$ has no SIL-pairs \cite{KobanPiggott}. On a second level, Day and Wade  gave a more general combinatorial condition on $\Gamma$ which determines whether the image $\POut(A_\Gamma)$ of $\PAut(A_\Gamma)$ in the group $\Out(A_\Gamma)$ of outer automorphisms is a RAAG or not; see \cite[Theorem~B]{DayWade2018}.

RAAGs have been analyzed frequently (see e.g. \cite{DuchampKrob}) through the $\N$-graded Lie algebra
\[\gr(A_{\Gamma}) = (\oplus_{n=1}^\infty \gamma_n(A_\Gamma)/\gamma_{n+1}(A_\Gamma))\otimes k\]
built from its lower central series and tensored with a field $k$ of characteristic $0$.
It turns out that these Lie algebras have the same commutator-relator presentations coming from the combinatorial structure of the graph as the RAAGs themselves, that is, they are the ``right-angled Artin Lie algebras'' associated with the graphs.

An important result in this perspective is the one by Froberg \cite{Froberg} showing that such Lie algebras are \emph{Koszul}: there is a linear graded projective resolution
\[ \cdots \to P_n \to P_{n-1} \to \cdots \to P_1 \to P_0 \to k \to 0\]
of the trivial $U(\gr(A_{\Gamma}))$-module $k$. Linear here means that the graded module $P_i$ is generated by its component of degree $i$. This is a classical niceness homological condition for the enveloping algebra $U(\gr(A_{\Gamma}))$.

Our first goal in this paper is to analyze the Koszul property for the $\N$-graded Lie algebras associated with $\PAut(A_{\Gamma})$. Apart from the cases when $\PAut(A_{\Gamma})$ is already a RAAG, so that their algebras are Koszul by Froberg's theorem, the first step comes from the result of Conner and Goetz \cite{ConnerGoetz}: $\gr(\PAut(F_n))$ is Koszul if and only if $n \leq 3$. This can be turned into a more general obstruction for Koszulness of $\gr(\PAut(A_{\Gamma}))$ by means of following combinatorial condition on the graph $\Gamma$:

\begin{itemize}
 \item[(*)] $\Gamma$ does not contain four pairwise non-adjacent vertices $v_1, v_2, v_3, v_4 \in \Gamma$ that lie in four distinct connected components of $\Gamma \smallsetminus \cap_{i=1}^4\lk(v_i)$
\end{itemize}

Here $\lk(v_i)$ denotes the link of the vertex $v_i$ in $\Gamma$.  We prove:

\begin{teo} \label{teoA}
The graded Lie algebra $\gr(\PAut(A_{\Gamma}))$ is Koszul if and only if $\Gamma$ satisfies condition (*).
\end{teo}

To prove the theorem above we need to work with a concrete presentation of $\gr(\PAut(A_{\Gamma}))$. The key to obtain it is the concept of \emph{formality}. Recall that a group $G$ is \emph{$1$-formal} if its Malcev completion is a quadratically presented, complete Lie algebra.

\begin{teo} \label{teoB}
For any finite graph $\Gamma$, the groups $\PAut(A_{\Gamma})$ and $\POut(A_{\Gamma})$ are $1$-formal. As a consequence, $\gr(\PAut(A_{\Gamma}))$ has a presentation where the generators correspond to the partial conjugations $c_L^v$, and the defining relators are
\begin{enumerate}
 \item $[c_L^v, c_K^w]=0$ if $v \in \st(w)$, or one of the components $L$ or $K$ is subordinate, or both $L$ and $K$ are shared but distinct, and
 \item $[c_L^v, c_L^w + c_K^w]=0$ when $L$ is shared and $K$ is dominant.
\end{enumerate}
Moreover, a presentation of $\gr(\POut(A_{\Gamma}))$ is obtained from the above by including the relations $\sum_L c_L^v$, where the sum runs through all connected components $L$ of $\Gamma \smallsetminus \st(v)$, for each $v \in \Gamma$.
\end{teo}

As we only need to apply $1$-formality to obtain the Lie algebra presentations, we will not need to enter into the details of the theory. A short sketch of what we need is in Section~\ref{sec.formality}; for more details see e.g. \cite{SuciuWang2}.

We note that a version of this result for the McCool groups is already considered in \cite{BerceanuPapadima}, and some related results about relations holding in the Lie algebra associated with $\PAut(F_n)$ were considered in \cite{Cohen} too.

The argument for Theorem~\ref{teoA} depends also on a certain decomposition of the given Lie algebras as iterated extensions of right-angled Artin Lie algebras. Such decomposition actually occurs already on the level of groups, which leads to the second main goal of this article.

Day and Wade developed in \cite{DayWade} a theory of relative automorphisms of RAAGs. They consider certain subgroups $\Out(A_{\Gamma}, \mathcal{G}, \mathcal{H}^t) \leq \Out(A_\Gamma)$ determined by outer automorphisms that preserve (resp. act trivially on) subgroups in a given family $\mathcal{G}$ (resp. $\mathcal{H}$) of special subgroups of $A_{\Gamma}$. Roughly, they prove that a certain finite-index subgroup $\Out^0(A_\Gamma)$ of $\Out(A_\Gamma)$ admits a subnormal series in which the successive quotients are the ``irreducible'' relative automorphism groups, which turn out to be either free abelian, isomorphic to some $\GL_n(\Z)$, or a Fouxe-Rabinovitch group.

We show that these decompositions, applied to our situation of pure symmetric automorphisms and under the hypothesis that condition (*) holds,  return subnormal series with factors being always RAAGs (in this paper, RAAG means \emph{finitely generated} RAAG). Below we denote
 \[\POut(A_\Gamma,\mathcal{G},\mathcal{H}^t)\coloneqq \Out^0(A_\Gamma,\mathcal{G}\cup\mathcal{C},(\mathcal{H}\cup\mathcal{C})^t),\]
where  $\mathcal{C} = \{ \langle v \rangle \mid v \in \Gamma\}$ is the family of cyclic special subgroups of $A_\Gamma$.

 \begin{teo} \label{teoC}
 Let $\mathcal{G}$ and $\mathcal{H}$ be sets of special subgroups of $A_\Gamma$. Then there is a subnormal series
 \[1=N_0\trianglelefteq N_1\trianglelefteq\cdots\trianglelefteq N_t=\POut(A_\Gamma,\mathcal{G},\mathcal{H}^t)\]
 such that each factor $N_i/N_{i-1}$ is either free abelian or a Fouxe-Rabinovitch group. If the graph $\Gamma$ has property (*), then these Fouxe-Rabinovitch groups are extensions of RAAGs, so $\POut(A_\Gamma,\mathcal{G},\mathcal{H}^t)$ is poly-RAAG.
 \end{teo}

 Notice that in particular, by taking $\mathcal{G} = \mathcal{H} = \mathcal{C}$, we obtain a subnormal series for $\POut(A_\Gamma)$ with RAAG factors, as soon as (*) holds. Moreover, since $\POut(A_\Gamma) = \PAut(A_\Gamma)/\Inn(A_\Gamma)$ and $\Inn(A_{\Gamma})$ is a RAAG,  we have the following Corollary:
 
 \begin{teo} Let $\Gamma$ be a graph with property (*). Then both $\POut(A_\Gamma)$ and $\PAut(A_\Gamma)$ are poly (finitely generated) RAAGs.
\end{teo}

 There has been a lot of interest recently about groups which admit (sub)normal series with \emph{free} factors. For instance all RAAGs admit such series: they are \emph{polyfree} in the sense that they admit subnormal series with (not necessarily finitely generated) free groups as factors \cite{DuchampKrob}. Notice that Theorem~\ref{teoC} implies that $\PAut(A_\Gamma)$ and $\POut(A_\Gamma)$ also admit such polyfree series.

 Another important aspect of Theorem~\ref{teoC} is that $N_{i+1}$ acts trivially on $N_i/N_i'$ for all $i$. Using this, we apply a theorem of Falk and Randell \cite{FalkRandell} to justify that we get a similar decomposition of the associated graded Lie algebras, which is used to prove Theorem~\ref{teoA}.

 The following remains unanswered:

 \begin{question*}
  If $\PAut(A_\Gamma)$ admits a subnormal series with (finitely generated) RAAGs as factors, does $\Gamma$ satisfy (*)?
 \end{question*}

As a partial result, we prove the following:

\begin{teo} \label{teoE}
 If $n \geq 4$, then $\PAut(F_n)$ and $\POut(F_n)$ are not poly-finitely generated free, that is, they do not admit  subnormal series with finitely generated free factors.
\end{teo}

This is more treatable than the general case because we can at the same time work cleanly with the cohomological dimension and the Euler characteristics of $\PAut(F_n)$ and of free groups, in contrast with arbitrary $\PAut(A_{\Gamma})$ and arbitrary RAAGs, respectively.


\section{Formality and the associated graded Lie algebra}  \label{sec.formality}
In this section we introduce and study the formality properties for $\PAut(A_{\Gamma})$ and $\POut(A_\Gamma)$, following mainly Berceanu and Papadima\cite{BerceanuPapadima} and Papadima and Suciu \cite{PapadimaSuciu}. This will lead to the proof of Theorem~\ref{teoB}.

Let $k$ be field of characteristic $0$. A \textit{Malcev Lie} $k$-\textit{algebra} is a Lie algebra $E$ over $k$ endowed with a complete descending filtration by subspaces $\{F^r(E)\}_r$ such that
\begin{enumerate}
 \item $F^1(E) = E$;
 \item $[F^r(E), F^s(E)] \subseteq F^{r+s}(E)$ for all $r,s$;
 \item The associated graded Lie algebra $\gr(E) = \oplus_{r \geq 1} F^{r}(E)/F^{r+1}(E)$ is generated in degree $1$.
\end{enumerate}
Here \emph{complete} means that the natural map $E \to \varprojlim E/F^r(E)$ is an isomorphism.

For each Malcev Lie algebra $E$ there is an associated exponential group $\exp(E)$. It is
the group defined on the underlying set of $E$ with product given by the Baker-Campbell-Hausdorff formula.  These two points of view are connected by the
exponential power series and the logarithm, which are bijections between $E$ and $\exp(E)$. 
Furthermore, the filtration of $E$ induces a group filtration in $\exp(E)$, and it can be shown that $\gr(E) \cong \gr(\exp(E))$. Notice that here we use $\gr(-)$ both as the graded algebra coming from a filtered Lie algebra and for the usual $\gr(G)$ coming from the lower central series of a group.

A \textit{Malcev completion} for a group $G$ is a homomorphism $\sigma \colon G \to \exp(E)$, for some Malcev Lie algebra $E$, such that
the induced homomorphism on the graded Lie algebras $\gr(\sigma) \colon \gr(G) \to \gr(E)$ is an isomorphism. For a free group $F_n$ with free basis $x_1, \ldots, x_n$,
the homomorphism $\sigma_n \colon F_n \to \exp(E_n)$, where $E_n$ is the free complete Lie algebra with free basis $\bar{x}_1, \ldots, \bar{x}_n$ and
$\sigma_n(x_i)= e^{\bar{x}_i}$, is a Malcev completion.

Let $G = \langle x_1, \ldots, x_n \mid r_1, \ldots, r_m \rangle$ be a finitely presented group. Let $E_G$ be the complete Lie algebra generated by
$\bar{x}_1, \ldots, \bar{x}_n$ and with defining relators $\log( \sigma_n(r_j) )$, for $1 \leq j \leq m$. By \cite[Thm.~2.2]{Papadima}, the homomorphism $\sigma \colon G \to \exp(E_G)$, induced by $\sigma_n$, is a Malcev completion.

We say that a group $G$ is $1$-\textit{formal} if $E_G$ is quadratic, that is, $E_G$ is the quotient of a free complete Lie algebra by a closed
ideal generated by homogeneous elements of degree $2$. This is equivalent with $\gr(G)$ being quadratic and $E_G$ being isomorphic as a filtered
algebra to the completion of its associated graded algebra (see \cite{SuciuWang2}, for instance).

In order to show that $\PAut(A_\Gamma)$ and $\POut(A_{\Gamma})$ are $1$-formal, we prove the following generalization of \cite[Theorem~5.4]{BerceanuPapadima}.

\begin{theorem}  \label{1formalthm}
 Let $G$ be a group admitting a finite presentation with set of generators $X$ and relators of the form $[x_1 \cdots x_m, y_1 \cdots y_n]$, with
 $x_i,y_i \in X$ and $m,n \geq 1$. Suppose further that if $[x_1 \cdots x_m, y_1 \cdots y_n]$ is a relator, then so are $[x_i,x_j]$
 and $[y_i,y_j]$ for all $i,j$ occurring in the expression. Then $G$ is $1$-formal.
\end{theorem}

\begin{proof}
To avoid confusions, we will denote by $\{\cdot, \cdot\}$ the group commutator on $\exp(E_G)$, thus keeping the usual brackets for the Lie algebra
operation on $E_G$.

Let $F$ be the free group with free basis $X$ and $E$ be the free complete Lie algebra with free basis $\bar{X} = \{\bar{x} \mid x \in X\}$. By the observations above, $E_G$ can be taken as the quotient of $E$ by the ideal generated by the elements $\log(\sigma(r))$, where $\sigma \colon F \to \exp(E)$ is the Malcev completion of $F$ and $r$ runs through the
defining relators of $G$.

Consider first a relator of the form $r=[x,y]$, with $x,y \in X$. Then $\log(\sigma(r))=\log(\{e^{\bar{x}},e^{\bar{y}}\}) \in E$. This is exactly the commutator of $\bar{x}$ and $\bar{y}$ when we see $\exp(E)$ as a group on the underlying set of $E$. By \cite[Lemma~2.5]{Papadima2}, killing this commutator
in $E$ amounts to the same as killing the bracket $[\bar{x},\bar{y}]$, now with usual operation of the Lie algebra $E$.

Now consider a relation of the form $r = [x_1 \cdots x_m, y_1 \cdots y_n]$ with $m > 1$ or $n > 1$, and $x_i,y_j \in X$. We have
$\log(\sigma(r))=\log(\{e^{\bar{x}_1} \cdots e^{\bar{x}_m}, e^{\bar{y}_1} \cdots e^{\bar{y}_n}\})$.
Again, killing $\log(\sigma(r))$ in $E$ is the same as killing the element $[\log(e^{\bar{x}_1} \cdots e^{\bar{x}_m}), \log(e^{\bar{y}_1} \cdots e^{\bar{y}_n})]$. Now, by hypothesis
$\log(\sigma([\bar{x}_i,\bar{x}_j]))$ and $\log(\sigma([\bar{y}_i,\bar{y}_j]))$ are relators for the appropriate indices. By the previous paragraph we may actually assume that $[\bar{x}_i,\bar{x}_j]$ and $[\bar{y}_i, \bar{y}_j]$ are relators.  For commuting elements $X,Y$ the Baker-Campbell-Hausdorff formula gives $e^X \cdot e^Y = e^{X+Y}$, so  we have:
\[\log(e^{\bar{x}_1} \cdots e^{\bar{x}_m}) = \log( e^{\bar{x}_1 + \ldots + \bar{x}_m} ) = \bar{x}_1 + \ldots + \bar{x}_m,\]
and similarly $\log(e^{\bar{y}_1} \cdots e^{\bar{y}_n}) = \bar{y}_1 + \ldots + \bar{y}_n$. So $\log(\sigma(r))$ is equivalent to the relator $[\bar{x}_1 + \ldots + \bar{x}_m, \bar{y}_1 + \ldots + \bar{y}_n]$.

All types of relators of $E_G$ are clearly homogeneous of degree $2$, so $G$ is $1$-formal.
\end{proof}

\begin{crlr}  \label{crlr1}
Under the same hypothesis of Theorem~\ref{1formalthm}, the graded Lie algebra $\gr(G)$ admits a presentation with $X$ as a generating set
and relators \[[x_1 + \ldots +x_m, y_1 + \ldots + y_n]\] whenever $[x_1 \cdots x_m, y_1 \cdots y_n]$ is a relator of $G$.
\end{crlr}

\begin{proof}
 This follows from the defining property of Malcev completions, that is, $\gr(G) \cong \gr(E_G)$.
\end{proof}

The theorem and its corollary apply immediately to our situation.
\begin{proof}[Proof of Theorem~\ref{teoB}]
  It is clear that the presentation of $\PAut(A_\Gamma)$ satisfies the conditions of Theorem~\ref{1formalthm} and Corollary~\ref{crlr1}.  Now, for $\POut(A_{\Gamma})$, we obtain a presentation by adding one relator of the form $\prod_L c_L^v$ for each $v \in \Gamma$ that is not connected to every other vertex of $\Gamma$. The product runs trough all connected components of
$\Gamma \smallsetminus \st(v)$. We may choose arbitrarily a connected component $K$ of $\Gamma \smallsetminus \st(v)$ for each
$v \in \Gamma$, and then eliminate the generators $c_K^v$ from the presentation by means of Tietze transformations. One easily sees now that the new presentation is equivalent to one containing only relators of the form $[\prod c_L^v , \prod c_N^w]$, for some products $\prod c_L^v$ and $\prod c_N^w$ of (commuting) conjugations by the same vertex.
This satisfies the conditions of Theorem~\ref{1formalthm} and Corollary~\ref{crlr1} too. By including again the generators $c_K^v$, we arrive at the presentation we wanted for $\gr(\POut(A_\Gamma))$.
 \end{proof}

 \section{Graded Lie algebras and Koszulness}
We consider here (Lie) algebras over $k$, a field of characteristic $0$. We say that a Lie algebra $\g$ is $\N$-graded if there is a vector space decomposition $\g = \oplus_{n=1}^\infty \g_n$ satisfying $[\g_n,\g_m] \subseteq \g_{n+m}$. A graded $\g$-module is a $\g$-module with a decomposition $A = \oplus_{n \in \Z} A_n$ satisfying $\g_n \cdot A_m \subseteq A_{m+n}$.

The universal enveloping algebra $U(\g)$ of an $\N$-graded Lie algebra admits a natural $\N_0$-grading. For any $\N$-graded Lie algebra we may consider the trivial module $k$ as a graded module with $k = k_0$. It makes sense then to consider the bigraded module
\[ \{\Ext_{U(\g)}^{p,q}(k,k)\}_{p,q}\]
where $p$ is the cohomological degree and $q$ is the internal degree, inherited from $\g$. We say that $\g$ is \emph{Koszul} if $\g$ is finitely generated by elements of degree one and and $\{\Ext_{U(\g)}^{p,q}(k,k)\}$ is \emph{diagonal}, that is, $\Ext_{U(\g)}^{p,q}(k,k) = 0$ whenever $p\neq q$. This is equivalent with saying that $U(\g)$ is Koszul in the sense of associative quadratic algebras (see e.g. \cite{Quadratic}).

Recall that an ideal $\h$ of a graded Lie algebra $\g = \oplus_{n=1}^\infty \g_n$ is \emph{homogeneous} if $\h = \oplus_{n=1}^\infty \h \cap \g_n$.

\begin{prop} \label{Koszul.prop}
 Let $\g$ be a graded Lie algebra and let $\h \subseteq \g$ be a homogeneous ideal.
 \begin{enumerate}
  \item If the quotient map $\g \to \g/\h$ splits and $\g$ is Koszul, then $\g/\h$ is Koszul.
  \item If $\h$ and $\g/\h$ are Koszul, then so is $\g$.
 \end{enumerate}
\end{prop}
\begin{proof}
 For item (1), assume that $\g$ is Koszul. So $\g$ is finitely generated, which implies that $\g/\h$ is finitely generated too. Now, let $\sigma \colon \g/\h \to \g$ be a splitting to the quotient map $\pi \colon \g \to \g/\h$. Then the induced graded maps
  \[ \sigma^{\ast} \colon \Ext_{U(\g)}^{\bullet,\bullet}(k,k) \to \Ext_{U(\g/\h)}^{\bullet,\bullet}(k,k)\  \ \text{  and  }\  \
  \pi^{\ast} \colon \Ext_{U(\g/\h)}^{\bullet,\bullet}(k,k) \to \Ext_{U(\g)}^{\bullet,\bullet}(k,k)\]
  preserve degree and satisfy $\sigma^\ast \circ \pi^\ast = \id$. In particular, $\Ext_{U(\g/\h)}^{\bullet,\bullet}(k,k)$ is a quotient of the diagonal bigraded module $\Ext_{U(\g)}^{\bullet,\bullet}(k,k)$, so it is diagonal too. This proves (1).

  Next, we consider item (2). Clearly $\g$ is finitely generated if both $\h$ and $\g/\h$ are. Assume that both $\Ext_{U(\h)}^{\bullet,\bullet}(k,k)$ and $ \Ext_{U(\g/\h)}^{\bullet,\bullet}(k,k)$ are diagonal. The associated graded LHS sequence is of the form
  \[ E_2^{\{p,q\},\bullet} =  \Ext_{U(\g/\h)}^{p,\bullet}(k, \Ext_{U(\h)}^{q,\bullet}(k,k)) \implies \Ext_{U(\g)}^{p+q,\bullet}(k,k),\]
 where the internal degree $\bullet$ is preserved. Notice that the action of $\g/\h$  increases
the internal degree, but preserves the cohomological one. Since $\Ext_{U(\h)}^{\bullet,\bullet}(k,k)$ is diagonal, we find that the action is trivial, so
    \[ E_2^{\{p,q\},\bullet}  \cong  \Ext_{U(\g/\h)}^{p,\bullet}(k,k) \otimes \Ext_{U(\h)}^{q,\bullet}(k,k) \cong \Ext_{U(\g/\h)}^{p,p}(k,k) \otimes \Ext_{U(\h)}^{q,q}(k,k).\]
 In particular, $E_2^{\{p,q\},n} = 0$ whenever $n \neq p+q$. From the convergence of the sequence we deduce that $\Ext_{U(\g)}^{m,n}(k,k)=0 $ whenever $m \neq n$, that is, $\Ext_{U(\g)}^{\bullet,\bullet}(k,k)$ is diagonal.
\end{proof}

\begin{prop}\label{polyRAAGLie} Let $\g$ be a finitely generated graded Lie algebra with a series of subalgebras
\[0=\g_0 \triangleleft \g_1 \triangleleft\ldots\triangleleft \g_{s-1}\triangleleft \g_s=\g\]
such that each $\g_i$ is an homogeneous ideal in $\g_{i+1}$ and the quotient $\g_{i+1}/\g_i$ is Koszul for all $i$.
Then $\g$ is Koszul.
\end{prop}
\begin{proof}
By induction, using Part (2) of Proposition~\ref{Koszul.prop}.
\end{proof}

Now, we specialize to the groups of pure symmetric automorphisms of RAAGs and their associated Lie algebras.

\begin{lemma} \label{l1}
 Suppose that a graph $\Gamma$ contains $n$ pairwise non-adjacent vertices $v_1, \ldots, v_n$ such that each $v_i$ lies in a distinct connected component of $\Gamma \smallsetminus \cap_{i=1}^n\lk(v_i)$. Then:
 \begin{enumerate}
  \item $\PAut(F_n)$ is a retract of $\PAut(A_\Gamma)$, and
  \item $\gr(\PAut(F_n))$ is a split quotient of $\gr(\PAut(A_{\Gamma}))$.
 \end{enumerate}
\end{lemma}
\begin{proof}
 Denote by $L_i$ the connected component of $\Gamma \smallsetminus \cap_{j=1}^n \lk(v_j)$ that contains $v_i$. Notice that
 $L_i$ is also a connected component of $\Gamma \smallsetminus \st(v_j)$ for $i \neq j$. Indeed, if $L$
 is the connected component of $\Gamma \smallsetminus \st(v_j)$ containing $v_i$, then clearly $L \subseteq L_i$. On the other hand, if $L \subsetneq L_i$, then $L_i \cap \st(v_j) \neq \emptyset$, which implies that $v_j \in L_i$ by connectedness. This is a contradiction.

 Let $\Phi \colon \PAut(A_\Gamma) \to \PAut(F_n)$ be the homomorphism that sends $c_{L_i}^{v_j}$ to $c_{ij}$ and sends all the other generators to $1$. Since the $c_{L_i}^{v_j}$ and the $c_{ij}$ satisfy the same set of relations, both $\Phi$ and its right-inverse $\Psi \colon \PAut(F_n)  \to \PAut(A_\Gamma)$ are well-defined. This proves part (1), and  part (2) is completely analogous using the presentation given in Theorem~\ref{teoB}.
\end{proof}

The hypothesis we make in Lemma~\ref{l1} for $n=4$ is the negation of condition (*) considered in the Introduction. We restate it for convenience.

\begin{definition}
 Let $\Gamma$ be a graph. We say that $\Gamma$ has (*) if $\Gamma$ does not contain four   pairwise non-adjacent vertices $v_1, v_2, v_3, v_4 \in \Gamma$ that lie in four distinct connected components of $\Gamma \smallsetminus \cap_{i=1}^4\lk(v_i)$
\end{definition}

\begin{prop}  \label{obstruction}
Suppose that the graph $\Gamma$ does not satisfy condition (*). Then
$\gr(\PAut(A_{\Gamma}))$ is not Koszul.
\end{prop}

\begin{proof}
 By applying Lemma~\ref{l1} we find that $\gr(\PAut(F_4))$ is a split quotient of $\gr(\PAut(A_{\Gamma}))$. Since  $\gr(\PAut(F_4))$ is not Koszul by \cite[Theorem~3.6]{ConnerGoetz}, we are done by Proposition~\ref{Koszul.prop}(1).
 \end{proof}

 This proves one direction of Theorem~\ref{teoA}. The other direction will come from a simultaneous analysis of the group and the Lie algebra by means of the subnormal series of Day and Wade.

\section{An obstruction to being poly-finitely generated free}
For the proof of Theorem~\ref{teoE}, we will need to recall some facts about the homological properties of the groups $\PAut(F_n)$.

Recall that the \emph{cohomological dimension} $\cd(G)$ of a group is the minimal length of a projective resolution of the trivial $\Z G$-module $\Z$. One may similarly consider the cohomological dimension $\cd_R(G)$ over a commutative ring $R$. If $\Z$ admits a finite-length resolution by finitely generated projective $\Z G$-modules, we say that $G$ is of \emph{type $\FP$}. Groups satisfying this property have a well-defined \emph{Euler characteristic } given by the alternating sum $\chi(G) = \sum (-1)^i \rk_\Z H_i(G,\Z)$; see \cite{BrownBook}.

 D.J. Collins \cite{Collins} proved that the McCool group $\PAut(F_n)$ is of type $\FP$ and has cohomological dimension $n-1$. Its Euler characteristic was then computed by J. McCammond and J. Meier \cite[Thm.~1.2]{McCammondMeier}:
\begin{equation} \label{euler}
\chi(\PAut(F_{n+1})) = (-1)^n n^n.
\end{equation}

Below we denote by $\X_j = \{ c_{i,j} \mid \forall 1 \leq i \leq n, \ \ i \neq j\}$ the subset of standard generators of $\PAut(F_n)$ that act with the same generator $x_j$ of $F_n$. Notice that $\langle \X_j \rangle \cong \Z^{n-1}$. In particular, any subgroup of $\PAut(F_n)$ containing some $\X_j$ has cohomological dimension $n-1$.

In the proof below we will use the BNS-invariant $\Sigma^1(G)$ of a finitely generated group \cite{BNS}. It is a set of homomorphisms $\varphi \colon G \to \R$ that contains the information of whether subgroups $N\leq G$ with $G' \leq N$ are finitely generated or not. We will need to use the basic fact that an element $\varphi \in \Sigma^1(G)$ induces $\bar{\varphi} \in \Sigma^1(G/N)$ if $\varphi(N) = 0$. Moreover, we will use that $\Sigma^1(F) = \emptyset$ if $F$ is free and non-cyclic.

\begin{lemma} \label{epi.free}
 Let $\pi \colon \PAut(F_n) \to F$ be  a homomorphism onto a free group, where $n \geq 4$. If at least $4$ of the $\X_j$ survive under $\pi$, then $F$ is cyclic.
 \end{lemma}

 \begin{proof}
  In what follows, we will use repeatedly the well-known fact that any set of commuting elements in a free group  generates a cyclic subgroup.

  Suppose for  a contradiction that $F$ is non-cyclic. Fix an embedding $\eta \colon F^{ab} \hookrightarrow \R$ and let $\varphi \colon \PAut(F_n) \to \R$ be the composition $\varphi = \eta \circ \mathrm{ab} \circ \pi$, where $\mathrm{ab} \colon F \to F^{ab}$ is the abelianization map.

  Since $\Sigma^1(F) = \emptyset$, we find that $\varphi \notin \Sigma^1(\PAut(F_n))$. By the description of $\Sigma^1(\PAut(F_n))$ in \cite{OrlandiKorner}, we may assume without loss of generality that $\varphi(c_{21}) \neq 0$ and that $\varphi$ kills all $c_{ij}$ except possibly when $i,j \in \{1,2,3\}$.   Notice that since $c_{21}$ commutes with $c_{i1}$ for all $i$, we find that $\pi(c_{21})$ and $\pi(c_{i1})$ are powers of a common element, and it can only be the case that $\pi(c_{i1})= 1$ if $i>3$.

  By hypothesis there is some $c_{ij}$, with $j \notin \{1,2,3\}$ that is not killed by $\pi$.
  We will find a contradiction by showing that $\pi(c_{ij})$ commutes with $\pi(c_{21})$, which means that they are powers of a common element in $F$, so in particular $\varphi(c_{ij}) = 0 $ if and only if $\varphi(c_{21}) = 0$.

  If $i \neq 1,2$, then $c_{ij}$ and $c_{21}$ already commute. If $i=2$, since we have a defining relation $[c_{2j}, c_{21}c_{j1}]=1$ and $\pi(c_{j1})= 1$, we find that $\pi(c_{2j})$ commutes with $\pi(c_{21})$. On the other hand, if $\pi(c_{2j}) =1$ but $\pi(c_{1j})\neq 1$, then the defining relation $[c_{21}, c_{1j}c_{2j}]=1$ determines similarly that $\pi(c_{21})$ commutes with $\pi(c_{1j})$.
  \end{proof}

\begin{proof}[Proof of Theorem~\ref{teoE}]
 Clearly it is enough to show that $G = \PAut(F_n)$ does not contain a proper normal subgroup $N$ such that i) $N$ is of type $\FP$ and ii) $G/N$ is free.

 Suppose for  a contradiction that such $N$ exists. By Feldman's theorem \cite[Thm.~5.5]{BieriBook} we have
 \[ \cd_\Q(N) <  \cd_\Q(G) = \cd_\Z(G) = n-1\]
 where the equalities hold because $G$ contains a copy of $\Z^{n-1}$. In particular $N$ does not contain any $\X_j$.

 By Lemma~\ref{epi.free}, the quotient $G/N$ must be infinite cyclic. But then the product formula for the Euler characteristics \cite[IX-7.3]{BrownBook} gives
 \[ \chi( \PAut(F_n) ) = \chi(N) \cdot \chi(G/N) = \chi(N) \cdot \chi(\Z) =0,\]
 which contradicts \eqref{euler}.
 \end{proof}

\section{Relative automorphisms of RAAGs} 
  In this section we mainly collect notation and results from Day and Wade \cite{DayWade} about relative automorphisms of RAAGs. This will be the framework for the rest of the paper, leading up to the proofs of Theorem~\ref{teoC} and the remaining part of Theorem~\ref{teoA}.
 
 Let $G$ be a group, $H\leq G$ a subgroup and $\Phi\in\Out(G)$. We say that $\Phi$ {\sl preserves} $H$ if there exists a representative $\phi\in\Phi$ such that $\phi(H)=H$ and we say that $\Phi$ {\sl acts trivially} on $H$ if there exists a representative $\phi\in\Phi$ such that $\phi$ is the identity on $H$. We use the same terminology for elements $\phi\in\Aut(G)$, saying that $H$ is preserved by $\phi$ if it is preserved by $\phi\Inn(G)\in\Out(G)$ and that $\phi$ acts trivially on $H$ if its coset in $\Out(G)$ does.

 Now, let $G=A_\Gamma$ be a right angled Artin group and let $\mathcal{G}$ and $\mathcal{H}$ be families of special subgroups of $A_\Gamma$ (i.e., subgroups generated by subgraphs of $\Gamma$). Let $\Out^0(A_\Gamma)$ be the subgroup of $\Out(A_\Gamma)$ generated by the cosets of inversions, partial conjugations and transvections  (see \cite{DayWade} for definitions). Day and Wade define $\Out^0(A_\Gamma,\mathcal{G},\mathcal{H}^t)$ as the subgroup of those $\Phi\in\Out^0(A_\Gamma)$ which preserve all the subgroups in $\mathcal{G}$ and act trivially on all the subgroups in $\mathcal{H}$.
 
 If $\Phi\in\POut(A_\Gamma)$, then for any $v\in\Gamma$ and any $\phi\in\Phi$ we have $v^\phi=v^{g_v}$ for some $g_v\in A_\Gamma$ so if $\phi_{g_v}$ is conjugation by $g_v$, $\alpha=\phi\phi_{g_v}^{-1}\in\Phi$ and $v^\alpha=v$ which implies that $\Phi$ acts trivially on the cyclic group $\langle v\rangle$ for any $v\in\Gamma$. Conversely, if $\Phi$ acts trivially on $\langle v\rangle$ for any $v\in\Gamma$, then for each $v\in\Gamma$ we may choose $\alpha_v\in\Phi$ with $v^{\alpha_v}=v$ so for any $\phi\in\Phi$ they differ by an inner automorphism, i.e., there is some $g_v\in A_\Gamma$ such that $v^\phi=(v^{\alpha_v})^{g_v}=v^{g_v}$, in other words, $\Phi\in\POut(A_\Gamma)$. This means that if we set
 \[\mathcal{C}=\{\langle v\rangle\mid v\in\Gamma\},\]
 then 
 \[\POut(A_\Gamma)=\Out^0(A_\Gamma,\mathcal{C},\mathcal{C}^t).\]
 From now on, we will use a variation of Day-Wade's notation and set
 \[\POut(A_\Gamma,\mathcal{G},\mathcal{H}^t)\coloneqq \Out^0(A_\Gamma,\mathcal{G}\cup\mathcal{C},(\mathcal{H}\cup\mathcal{C})^t).\]

 Following \cite{DayWade}, we say that the family $\mathcal{G}$ is {\sl saturated} with respect to $(\mathcal{G},\mathcal{H})$ if $\mathcal{G}$ contains every proper special subgroup which is  preserved by $\Out^0(A_\Gamma,\mathcal{G},\mathcal{H}^t)$. To avoid complications, unless otherwise stated, we will assume that this is always the case. Note that there is no loss of generality in doing that, since
 \[\Out^0(A_\Gamma,\overline{\mathcal{G}},\mathcal{H}^t)=\Out^0(A_\Gamma,\mathcal{G},\mathcal{H}^t)\]
 for $\overline{\mathcal{G}}$ the set of proper special subgroups preserved by $\Out^0(A_\Gamma,\mathcal{G},\mathcal{H}^t)$.
 
 We will also denote by $\overline{\mathcal{C}}$ the set of all proper special subgroups preserved by $\POut(A_\Gamma)$, in other words, the set of all proper special subgroups which are preserved by any partial conjugation.
  
Next, we state  \cite[Theorem E]{DayWade} with the notation established above. Let $\Delta \subseteq\Gamma$ such that $A_\Delta\in\mathcal{G}$. Then there is a short exact sequence
 \[1\to\POut(A_\Gamma,\mathcal{G},(\mathcal{H}\cup\{A_\Delta\})^t)\to\POut(A_\Gamma,\mathcal{G},\mathcal{H}^t)\to\POut(A_\Delta,\mathcal{G}_\Delta,\mathcal{H}_\Delta^t)\to 1\]
 where 
 \[\mathcal{G}_\Delta=\{A_{\Delta\cap\Theta}\mid A_\theta\in\mathcal{G}\}\smallsetminus\{A_\Delta\}\]
  and $\mathcal{H}_\Delta^t$ is defined analogously.
  
  Day and Wade iterate this sequence to obtain a subnormal series for the groups $\Out^0(A_\Gamma,\mathcal{G},\mathcal{H}^t)$ whose factors are either free abelian, or linear groups $\mathrm{GL}(m,\Z)$ or {\sl Fouxe-Rabinovitch} groups (defined below). We show in Section \ref{polyRAAG} that in the particular case when the group is $\POut(A_\Gamma)$ and the graph $\Gamma$ satisfies the previously defined condition (*), Day-Wade's procedure determines a subnormal series
  \[1=N_0\trianglelefteq N_1\trianglelefteq\ldots\trianglelefteq N_s=\POut(A_\Gamma)\]
  whose factors $N_i/N_{i-1}$ are all (finitely generated) RAAGs and moreover that each  $N_i/N_{i-1}$ acts trivially on the abelianization $N_{i-1}/N_{i-1}'$ of the next term.
  
  Later, we will refine that series to get a polyfree series. The extra condition of being polyfree will allow us to deduce that the partial short exact sequences that we get split and then use \cite{FalkRandell} to prove that there is a similar structure at the level of the associated Lie algebras.
  
  These constructions lift to a subnormal series of $\PAut(A_\Gamma)$ with similar properties by means of the exact sequence
  \[ 1 \to \Inn(A_\Gamma) \to \PAut(A_\Gamma) \to \POut(A_\Gamma) \to 1.\]
  Note that $\Inn(A_\Gamma)$ is isomorphic to a RAAG defined on the full subgraph $\Gamma_0 \subseteq \Gamma$ of the vertices that are not adjacent to every vertex of $\Gamma$.

  We will use the following result which follows from \cite[Theorem D]{DayWade}.
  
  \begin{lemma}\label{lem:generating} The group $G=\POut(A_\Gamma,\mathcal{G},\mathcal{H}^t)$ is generated by the partial conjugations that it contains.
   \end{lemma}
   
   The following technical lemma will also be useful.
   
   \begin{lemma}(\cite[Lemma 2.2(3)]{DayWade})\label{DayWade2.2}
   Let $\Delta\subseteq\Gamma$, $v\in\Gamma$, $K\subseteq\Gamma\smallsetminus\st_\Gamma(v)$ a union of connected components and $g$ the partial conjugation that conjugates vertices in $K$ by $v$ and fixes the rest.  Consider the following conditions
   \begin{itemize}
   \item[i)] $K\cap\Delta=\emptyset$,
   \item[ii)] $\Delta\smallsetminus\st_\Gamma(v)\subseteq K$,
   \item[iii)] $v\in\Delta$.
   \end{itemize}
   Then $g$ preserves $A_\Delta$ if and only if one of i), ii) or iii) happens and $g$ acts trivially on $A_\Delta$ if and only if one of i), ii) happens.
   \end{lemma}

   Using the results above,  we deduce the following.
   
   \begin{lemma}\label{condinv} Let $\Delta\subseteq\Gamma$ be such that $\Delta\smallsetminus\st_\Gamma(v)$ is connected for any $v\in\Gamma\smallsetminus\Delta$. Then $A_\Delta$ is preserved by any partial conjugation, i.e. $A_\Delta\in\overline{\mathcal{C}}$.
   \end{lemma}
   \begin{proof} Let $g$ be a partial conjugation that conjugates by the vertex $v$ all the vertices in $K\subseteq\Gamma$ and assume $v\not\in\Delta$. We may assume that $K$ is a union of connected components of $\Gamma\smallsetminus\st_\Gamma(v)$. As $\Delta\smallsetminus\st_\Gamma(v)$ is connected this means that either $\Delta\smallsetminus\st_\Gamma(v)\subseteq K$ or $K\cap\Delta=\emptyset$. So it is enough to use Lemma \ref{DayWade2.2}.
   \end{proof}

 \section{A poly-RAAG series}\label{polyRAAG}
  
We will now impose condition (*) on $\Gamma$ and see that Day and Wade's construction produces a subnormal series for $\POut(A_\Gamma)$ without factors of type $\mathrm{GL}(n,\Z)$, which will prove Theorem~\ref{teoC}.

The following result is a consequence of \cite[Corollary 3.12]{DayWade} but we include a proof for completeness.

\begin{lemma}\label{subgraph}
 Let $\Theta \subseteq \Delta \subseteq \Gamma$, and let $\mathcal{G}$ and $\mathcal{H}$ be families of proper special subgroups of $A_\Gamma$. If $A_\Delta$ is preserved by $\POut(A_\Gamma,\mathcal{G}, \mathcal{H}^t)$ and $A_\Theta$ is preserved by $\POut(A_\Delta,\mathcal{G}_\Delta, \mathcal{H}_\Delta^t)$, then $A_\Theta$ is also preserved by $\POut(A_\Gamma, \mathcal{G}, \mathcal{H}^t)$.
\end{lemma}
\begin{proof} Let $g$ be a representative for $\Phi\in\POut(A_\Gamma,\mathcal{G}, \mathcal{H}^t)$. As $A_\Delta$ is preserved by $\POut(A_\Gamma,\mathcal{G}, \mathcal{H}^t)$, there is some $x\in A_\Gamma$ such that $(w^g)^x\in A_\Delta$ for any $w\in\Delta$, so $g\rho_x\in\PAut(A_\Delta,\mathcal{G}_\Delta, \mathcal{H}_\Delta^t)$ where $\rho_x$ is conjugation by $x$. As $A_\Theta$ is preserved by $\POut(A_\Delta,\mathcal{G}_\Delta, \mathcal{H}_\Delta^t)$ there is some $y\in A_\Delta\subseteq A_\Gamma$ such that $(u^g)^{xy}\in A_\Theta$ for any $u\in\Theta$. 
\end{proof}

 We will need the following definition from \cite[3.3.1]{DayWade}.
 Let $\mathcal{G}$ be a family of proper special subgroups of $A_\Gamma$. We say that two vertices $u,v\in\Gamma$ are {\sl $\mathcal{G}$-adjacent} if either they are  adjacent in $\Gamma$ or there is some $A_\Delta\in\mathcal{G}$ such that $u,v\in\Delta$. A {\sl $\mathcal{G}$-path} is a finite sequence of vertices such that each vertex in the sequence is $\mathcal{G}$-adjacent to the next and a full subgraph of $\Delta\subseteq\Gamma$ is {\sl $\mathcal{G}$-connected} if for any two vertices $u,v\in\Delta$ there is a $\mathcal{G}$-path from $u$ to $v$. Maximal $\mathcal{G}$-connected subgraphs of  a subgraph $\Delta$ are called {\sl $\mathcal{G}$-connected components of $\Delta$}.
 
 Recall that we are denoting by $\overline{\mathcal{C}}$ the set of special subgroups which are preserved by $\POut(A_\Gamma)$. For a subgraph $\Delta\subseteq\Gamma$, we denote
 \[\overline{\mathcal{C}}_\Delta=\{A_{\Delta\cap\Theta}\mid A_\Theta\in \overline{\mathcal{C}}, \Delta\not\subseteq\Theta\}.\]

  Condition (*) does not pass to subgraphs, not even to preserved subgraphs, so we need a version of this that behaves well throughout the inductive process.

\begin{lemma}\label{connectedcomponents*} If $\Gamma$ is a graph that satisfies (*) and $A_\Delta\subseteq A_\Gamma$ is preserved by $\POut(A_\Gamma)$, then $\Delta$ has at most 3 $\overline{\mathcal{C}}_\Delta$-connected components.
\end{lemma}
\begin{proof} Assume, for a contradiction, that there are $v_1,v_2,v_2,v_4\in\Delta$ lying in different $\overline{\mathcal{C}}_\Delta$-connected components. In particular this implies that the $v_j$ are pairwise non-adjacent. 
Applying condition (*) we see that at least two of the vertices $v_1,\ldots,v_4$ are connected in  $\Gamma\smallsetminus\cap_{i=1}^4\lk(v_i)$. Let $p$ be a path of minimal length in $\Gamma\smallsetminus\cap_{i=1}^4\lk(v_i)$ joining two of the $v_i$, say $v_1$ and $v_2$. In particular neither $v_3$ nor $v_4$ are in $p$. Denote by $\Omega$ the subgraph induced by the following vertices
\[\cup\{\st_\Gamma(u)\mid u\in p,u\neq v_1,v_2\}\cup\{v_1,v_2\}.\]
Note that for every pair of vertices $u,v\in\Omega$ there is a path joining them lying entirely in $p$ (except of possibly the initial and terminal vertices  $u$ and $v$ themselves).
Let $w\in\Gamma\smallsetminus\Omega$. Then $p\cap\st_\Gamma(w)=\emptyset$ so we deduce that $\Omega\smallsetminus\st_\Gamma(w)$ is connected. Therefore Lemma \ref{condinv} implies that $A_\Omega$ is preserved by $\POut(A_\Gamma)$, i.e., $A_\Omega\in\overline{\mathcal{C}}$. Then either $A_{\Omega\cap\Delta}\in\overline{\mathcal{C}}_\Delta$ or $\Delta\subseteq\Omega$. In the first case we deduce that $v_1$ and $v_2$ are in the same $\overline{\mathcal{C}}_\Delta$-connected component so we get the result. So we assume $\Delta\subseteq\Omega$  and as $v_3,v_4\in\Delta$, they both lie in $\Omega$ (but not in $p$). Therefore we may choose $z\in p$ such that $z$ is linked to, say $v_3$ and then the minimality of $p$ implies that $z$ is also linked to $v_1$ or $v_2$. But then as there is a path from $v_3$ of length 2, we deduce that also $p$ has length 2 so $z$ is indeed linked to the four elements $v_1,v_2,v_3,v_4$ which contradicts the fact that $p$ is contained in $\Gamma\smallsetminus\cap_{i=1}^4\lk(v_i)$.
\end{proof}
 
Consider a group $S$ which decomposes as a free product
\[S=G_1\star\cdots\star G_m\]
 The {\sl Outer Fouxe-Rabinovitch} group of $S$ is the subgroup of $\Out(S)$ consisting on elements having, for each $i$, a representative acting as conjugation on $G_i$. In particular, if $S=A_\Gamma$ and $\Gamma$ is the disjoint union of $\Delta_i$, $i=1,\ldots,m$, then the Outer Fouxe-Rabinovitch group of $A_\Gamma$ is
 \[\mathrm{OFR}(A_\Gamma)=\POut(A_\Gamma, \mathcal{H}, \mathcal{H}^t).\]
 for $\mathcal{H} = \{A_{\Delta_1},\ldots,A_{\Delta_m}\}$.

Coming back to the general case $S=G_1\star\cdots\star G_m$, if we set
  \[\mathrm{FR}(S)=\{\phi\in\Aut(S)\mid \phi|_{G_i} \text{ is conjugation with some } x_i \in S, \text{ for }  i = 1, \ldots, m\},\]
 then $\mathrm{OFR}(S)=\mathrm{FR}(S)/\Inn(S)$. We will call $\mathrm{FR}(S)$ the {\sl Fouxe-Rabinovitch group} associated to the decomposition $S=G_1\star\cdots\star G_m$.
 
  \begin{remark} In \cite{DayWade} there is a more general definition of Fouxe-Rabinovitch used for groups with a decomposition   $G_1\star\cdots\star G_m\star\mathrm{Free}(k)$ but we will just need the restricted version above.
  \end{remark}
  
 Using the obvious projection $S=G_1\star\cdots\star G_m\to G_1\times\cdots\times G_m$ one sees that there is a homomorphism $\mathrm{FR}(S)\to\Aut(G_1\times\cdots\times G_m)$. Let $N$ be its kernel. 
  Consider also the subgroup
   \[H=\{\varphi\in\mathrm{FR}(S)\mid G_i^\varphi=G_i, i=1,\ldots,m\},\]
  observe that $N\cap H=1$ so we may see $H$ as a subgroup of $\Aut(G_1\times\cdots\times G_m)$, indeed
  $H\cong\Inn(G_1)\times\cdots\times\Inn(G_m).$
  
   A {\sl Whitehead automorphism} of $S=G_1\star\cdots\star G_m$ is a partial conjugation of the following form: fix  $j$ with $1\leq j\leq m$ and let $X_j=\{x_1,\ldots,x_m\}\subseteq G_j$. Then for each $g\in G_i$, $\varphi_{X_j}$ maps $g$ onto $g^{x_i}$. Note that in the case when $x_j=1$, then $\varphi_{X_j}\in N$, and in the case when all $x_i=1$ for $i\neq j$, $\varphi_{X_j}\in H$. The following Lemma follows from the previous observations together with \cite[Lemma 2.1]{McCulloughMiller}.
   
    \begin{lemma}\label{FRdecomposition} If $S=G_1\star\cdots\star G_m$, any element in the associated Fouxe-Rabinovitch group $\mathrm{FR}(S)$ can be written as $\varphi\psi$ where $\varphi$ is a product of Whitehead automorphisms and $\psi\in H$.
\end{lemma} 

  This has the following consequence:

 \begin{lemma}\label{FouRab}(\cite[Section 2.3]{McCulloughMiller}) The outer Fouxe-Rabinovitch group of a free product of 2  components is 
\[\mathrm{OFR}(G_1\star G_2)=\Inn(G_1)\times\Inn(G_2)\]
 and for 3  components we have
\[ \mathrm{OFR}(G_1\star G_2\star G_3)=Z\rtimes H\]
with $Z=G_1\star G_2\star G_3$ and $H=\Inn(G_1)\times\Inn(G_2)\times\Inn(G_3)$. Moreover, $H$ acts trivially on the abelianization $Z/Z'$.
  \end{lemma}
 \begin{proof} This is in \cite[Section 2.3]{McCulloughMiller} but we recall here part of the argument. Let $m$ be 2 or 3.  Let $K=\Inn(S)$ so that $\mathrm{OFR}(S)=\mathrm{FR}(S)/K$ and observe that $H\cap K=1$.

 Consider first the case when $m=2$, i.e. $S=G_1\star G_2$. Then for any Whitehead automorphism  $\varphi_{X_1}$ with $X_1=\{x_1,x_2\}\subseteq G_1$ we have
 $\varphi_{X_1}=\rho\varphi_{Y_1}$
 where $\rho\in K$ is conjugation by $x_2$ and $Y_1=\{x_2^{-1}x_1,1\}$. Then $\varphi_{Y_1}\in H$ thus $\varphi_{X_1}\in KH$. By symmetry we deduce the same for Whitehead automorphisms $\varphi_{X_2}$ with $X_2\subseteq G_2$. So we deduce $\mathrm{FR}(G_1\star G_2)=HK$ thus
 \[\mathrm{OFR}(G_1\star G_2)=HK/K\cong H\cong\Inn(G_1)\times\Inn(G_2).\]
  
  In \cite[Section 2]{McCulloughMiller} there is a construction of a simplicial complex on which the group  $\mathrm{OFR}(S)$ acts. In the case when $m=3$ this complex is a tree that yields the following graph of groups.
 
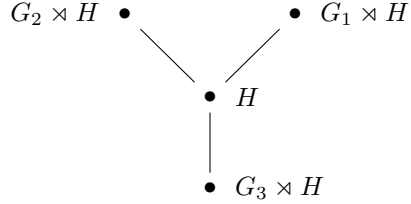
\[\begin{tikzpicture}
	\node[label=right:{$H$}] (1) {$\bullet$};
\node[label=right:{$G_1\rtimes H$}] (2) [above right=1cm  of 1] {$\bullet$}; 
\node[label=left:{$G_2\rtimes H$}] (3) [above left=1cm  of 1] {$\bullet$}; 
\node[label=right:{$G_3\rtimes H$}] (4) [below=0.8cm  of 1] {$\bullet$}; 
	\draw[-] (1) --  (2);
	\draw[-] (1) -- (3);
	\draw[-] (1) -- (4);
	\end{tikzpicture}\]
 For $i=1$ the vertex group labelled $G_1\rtimes H$ is the homomorphic image in $\mathrm{OFR}(S)$ of the group generated by $K$, $H$ and Whitehead automorphisms of the form $\varphi_{\{(1,x,1)\}}$, $x\in G_1$. It is easy to check that this group is indeed isomorphic to the semidirect product $G_1\rtimes H$. The vertex groups $G_2\rtimes H$ and $G_3\rtimes H$ are defined analogously. Using this graph we get the description of $\mathrm{OFR}(S)$ in the statement.
 The action of $H$ on $Z$ is by conjugation on the corresponding component, for example, if  $a,x\in G_1$:
 \[\varphi^{-1}_{\{(a,1,1)\}}\varphi_{\{(1,x,1)\}}\varphi_{\{(a,1,1)\}}=\varphi_{\{(1,x^a,1)\}}\]
 and if $b\in G_2$, $c\in G_3$ we have
  \[\varphi^{-1}_{\{(1,b,1)\}}\varphi_{\{(1,x,1)\}}\varphi_{\{(1,b,1)\}}=\varphi_{\{(1,x,1)\}},\]
 \[\varphi^{-1}_{\{(1,1,c)\}}\varphi_{\{(1,x,1)\}}\varphi_{\{(1,1,c)\}}=\varphi_{\{(1,x,1)\}}.\]
  \end{proof}

  \begin{proof}[Proof of Theorem~\ref{teoC}]
 We follow the proof of Theorem 5.9 in \cite{DayWade} and argue by induction on the pairs $(n,m)$ ordered lexicographically where $n=|\Gamma|$ and $m=2^n-r$ for $r$ the number of special subgroups of $A_\Gamma$ on which $\POut(A_\Gamma,\mathcal{G},\mathcal{H}^t)$ acts trivially. As Day-Wade do, we call this pair $(n,m)$ the {\sl complexity} of $\POut(A_\Gamma,\mathcal{G},\mathcal{H}^t)$.
 Then we may distinguish two possible cases:
 
 \medskip

 \noindent{\bf Case 1:} There is some $A_\Delta\in\mathcal{G}$ such that $\POut(A_\Gamma,\mathcal{G},\mathcal{H}^t)$ does not act trivially on $A_\Delta$. Then we have the short exact sequence
 \[1\to\POut(A_\Gamma,\mathcal{G},(\mathcal{H}\cup\{A_\Delta\})^t)\to\POut(A_\Gamma,\mathcal{G},\mathcal{H}^t)\to\POut(A_\Delta,\mathcal{G}_\Delta,\mathcal{H}_\Delta^t)\to 1.\]
  Both $\POut(A_\Gamma,\mathcal{G},(\mathcal{H}\cup\{A_\Delta\})^t)$ and $\POut(A_\Delta,\mathcal{G}_\Delta,\mathcal{H}_\Delta^t)$ have strictly smaller complexity than $\POut(A_\Gamma,\mathcal{G},\mathcal{H}^t)$ so by induction we get series for both which combined give a series for $\POut(A_\Gamma,\mathcal{G},\mathcal{H}^t)$.
 
 \medskip
 \noindent{\bf Case 2:} There is no such a $\Delta$ which means that
 \[G=\POut(A_\Gamma,\mathcal{G},\mathcal{H}^t)=\POut(A_\Gamma,\mathcal{G},\mathcal{G}^t).\]
 Recall that we are assuming that $\mathcal{G}$ is saturated. There are five subcases to consider, each described in \cite[Section 5]{DayWade}. In the sequence we follow their considerations (and subsection numbering), adapting to our case of pure symmetric automorphisms and imposing condition (*).

\medskip
 \noindent{\bf Case 5.1.1:} $\Gamma$ is $\mathcal{G}$-disconnected. Here $G$ is a Fouxe-Rabinovitch group. More precisely,
if we denote by $\Delta_1,\ldots,\Delta_t$ the $\mathcal{G}$-components of $\Gamma$ then $\Gamma$ is the disjoint union of $\Delta_1,\ldots,\Delta_t$ and therefore
\[A_\Gamma=A_{\Delta_1}\star\cdots\star A_{\Delta_t}.\]
Note that in the decomposition in \cite[Subsection~5.1.1]{DayWade} there is also a free factor consisting of vertices that do not lie in any group in $\mathcal{G}$, but in our case $\mathcal{G}$ contains the cyclic subgroups generated by any vertex, so that is never the case. In the case when the ambient graph has (*), we can not assume that our $\Gamma$ has also (*) because we are arguing by induction and (*) does not pass to subgraphs, but as by Lemma \ref{subgraph} all the subgraphs that we are considering are $\POut$-invariant, Lemma \ref{connectedcomponents*} implies that we can assume that $\Gamma$ has at most three $\mathcal{G}$-connected components, i.e. $2 \leq t\leq 3$. Lemma \ref{FouRab} implies that
\begin{itemize}
\item either $t=2$ and $G$ is the product of the inner automorphisms groups of $A_{\Delta_1}$ and $A_{\Delta_2}$ which are both RAAGs,

\item or $t=3$ and $G$ is a semidirect product of two RAAGs $Z\rtimes H$.
\end{itemize}

\medskip
 \noindent{\bf Case 5.1.2:} $\Gamma$ disconnected and $\mathcal{G}$-connected. In this case $G$ is free abelian, so we are done.

\medskip
 \noindent{\bf Case 5.1.3:} $\Gamma$ connected and $Z(A_\Gamma)=1$. Again $G$ is free abelian.

\medskip
 \noindent{\bf Case 5.1.4:} $\Gamma$ connected and $1\neq Z(A_\Gamma)\subsetneq\Gamma$. In this case, let $Z\subseteq\Gamma$ be such that $A_Z=Z(A_\Gamma)$ and $\Delta=\Gamma\smallsetminus Z$, \cite[Proposition 5.6]{DayWade} implies that there is a surjection
\[G\to\POut(A_\Delta,\mathcal{G}_\Delta\mathcal{G}^t_\Delta)\]
 with free abelian kernel. Moreover, according to the proof of that result, the kernel is generated by transvections which means that in our case the kernel is trivial so the surjection is an isomorphism.  As the complexity of $\POut(A_\Delta,\mathcal{G}^t_\Delta)$ is strictly smaller than that of $G$, we can argue by induction.

\medskip
 \noindent{\bf Case 5.1.5:} $\Gamma$ is a complete graph. In this case $A_\Gamma$ is free abelian, so that the associated automorphism group has no partial conjugations, i.e. $G=1$.
 \end{proof}
  
\section{A presentation for the groups in the Day-Wade series}\label{sec:presentation}
The objective of this section is to obtain a presentation for the factor groups in the Day-Wade series, both for the $\PAut$ and $\POut$ versions. This presentation is similar to the standard presentation for $\PAut(A_\Gamma)$ computed by Toinet \cite{Toinet}.

To describe these presentations, we extend previously used notation  to include partial conjugations of the form $c_A^v\in\PAut(A_\Gamma)$ where $A$ is a union of connected components of $\Gamma\smallsetminus\st_\Gamma(v)$. In this case we say that $c_A^v$ is a partial conjugation by $v$ and based at $A$.
Note that if $A$ and $B$ are disjoint, then $c_A^vc_B^v=c_{A\cup B}^v$.

The proof of following lemma is straightforward.

\begin{lemma}\label{lem:relators} Let $v,w\in\Gamma$ be not linked and let $c^v_A$, $c^w_B$ be partial conjugations.  Then
$[c_A^v,c_B^w]=1$ if and only if one of the following condition holds
\begin{itemize}
\item[i)] $A\cap B=\emptyset$, $v\not\in B$ and $w\not\in A$,
\item[ii)] $A\subseteq B$ and $v\in B$,
\item[iii)] $B\subseteq A$ and $w\in A$.
\end{itemize}
\end{lemma}

Another extension of notation is needed. Assume that we are given, for each $v \in \Gamma$, a partition $\Omega_G^v$ of $\Gamma \smallsetminus \st_\Gamma(v)$ into  \emph{unions} of connected components. For each pair of non-adjacent vertices $v,w\in\Gamma$, the subsets $D_w^v\in\Omega_G^v$ and $D_v^w\in\Omega_G^w$ with $v\in D_v^w$ and $w\in D_w^v$ will be called {\sl dominant}. Similarly, the elements $A\in\Omega_G^v$ such that $A\subseteq D_v^w$, and $B \in \Omega_G^w$ such that $B \subseteq D_w^v$ will be called {\sl subordinate}, and if $A$ lies in both $\Omega_v^G$ and $\Omega_w^G$, then we will say that $A$ is {\sl shared}. Again, these notions are all relative to the pair $(v,w)$, and here also depend on the partitions $\{\Omega_G^v\}_v$. Note that at this point we are not assuming that arbitrary elements  of $\Omega_G^w$, $\Omega_G^v$ can be only dominant, shared or subordinate, but this will play a role in Definition~\ref{def:PAut}  below.

In view of Lemma~\ref{lem:relators} and  this terminology, for $v,w \in \Gamma$ we will denote by $R^{v,w}_G$ the set of relators given by:
\begin{itemize}
\item $[c^v_A,c^w_B]=1$ if
\begin{enumerate}
 \item[(i)] either $v \in \st(w)$, with $A \in \Omega_G^v$ and $B \in \Omega_G^w$, or
 \item[(ii)] at least one of $A$, $B$ is subordinate, or $A\neq B$ are both shared, with $A \in \Omega_G^v$ and $B \in \Omega_G^w$.
\end{enumerate}
\item $[c_A^v,c_A^w c_{D_v^w}^w]=1$ and $[c_A^v c_{D_w^v}^v,c_A^w]=1$  if $v \notin \st(w)$,  $A \in \Omega_G^v \cap \Omega_G^v$ (that is, $A$ is shared), and $D_v^w \in \Omega_G^w$ and $D_w^v \in \Omega_G^v$ are dominant.
\end{itemize}

We will use these relators to define the concept of {\sl $\PAut$-like groups}. These are certain subgroups $G \leq \PAut(\Gamma)$ which admit presentations where generators are some partial conjugations in the extended sense and all relations come from Lemma~\ref{lem:relators}, in the same way as $\PAut(A_\Gamma)$ itself.

 \begin{definition}\label{def:PAut}
Let $G\leq\PAut(A_\Gamma)$ be a subgroup. Suppose that for each $v\in\Gamma$ there is a partition
$\Omega_G^v$ of $\Gamma\smallsetminus\st_\Gamma(v)$ into unions of connected components such that
\begin{itemize}
\item[i)] the partial conjugations $c_A^v$, with $A\in\Omega_G^v$ and $v\in\Gamma$, generate $G$,

\item[ii)] for each pair of not linked vertices $v,w\in\Gamma$, the non-dominant elements of $\Omega_G^v$ and $\Omega_G^w$ are either subordinate or shared relative to $\{\Omega_G^v\}_v$.
\end{itemize}
We say that $G$ is \emph{$\PAut$-like} if the sets $R_G^{v,w}$ of relators yield a presentation, i.e., if
\[G=\langle \Omega^v_G,v\in\Gamma\mid R^{v,w}_G, v,w\in\Gamma\rangle.\]
\end{definition}

In the situation above, we will call the partial conjugations $c_A^v$, with $A \in \Omega_G^v$, the {\sl standard generators} for $G$ and the elements of $R_G^{v,w}$ the {\sl standard relators}.

\begin{remark} Note that by definition any $\PAut$-like group contains $\Inn(A_\Gamma)$.
We will say that a subgroup $\overline{G}\leq\POut(A_\Gamma)$ is {\sl $\POut$-like} if there is a $\PAut$-like group $G\leq\PAut(A_\Gamma)$ such that $\overline{G}=G/\Inn(A_\Gamma)$. If $\Inn(A_\Gamma)\leq G\leq\PAut(A_\Gamma)$, then $G$ is $\PAut$-like if and only if $\overline{G} = G/\Inn(A_\Gamma)$ is $\POut$-like.
\end{remark}

In the rest of this section, we show that the groups associated to a Day-Wade series are all $\PAut$- and $\POut$-like. More precisely, write $G = \PAut(A_\Gamma, \mathcal{G}, \mathcal{H}^t)$ and $\overline{G}=\POut(A_\Gamma,\mathcal{G},\mathcal{H}^t)$,
and assume that $\overline{G}$ is one of the groups appearing in the series for $\POut(A_{\Gamma_1})$, where $\Gamma_1$ is a finite graph.  Here $\Gamma$ is a subgraph of $\Gamma_1$. We will check that if $G$ is $\PAut$-like, then for $A_\Delta \in \mathcal{G}$, the groups
\[H=\PAut(A_\Gamma,\mathcal{G},(\mathcal{H}\cup\{A_\Delta\})^t)\]
and
\[P=\PAut(A_\Delta,\mathcal{G}_\Delta,\mathcal{H}_\Delta^t)\]
are $\PAut$-like subgroups with respect to $\Gamma$ and $\Delta$, respectively.
 From this it will follow that the $\POut$ version \[Q = \POut(A_\Delta,\mathcal{G}_\Delta,\mathcal{H}_\Delta^t)\] is $\POut$-like as well. As we start with $\PAut(A_{\Gamma_1})$, which is trivially $\PAut$-like, the claim for all factors of the subnormal series will follow by induction.

 Note that the groups above fit into a short exact sequence
 \[1\to H\to G\to Q\to 1.\]
Observe that if $\Delta$ was complete, this would be trivial, as then $Q=1$. In fact, we will assume that there is no vertex $v \in \Delta$ such that $\Delta \subseteq \st_\Gamma(v)$, i.e., linked to all the rest of vertices of $\Delta$. We can assume this because otherwise removing $v$ from $\Delta$ yields another invariant subgraph with the same short exact sequence.

The argument will be as follows. We will define abstract groups $\hat H$ and $\hat P$ having the expected $\PAut$-like presentations and show that one can form a semidirect product
\[\hat G=\hat P\ltimes\hat H\]
such that there is a normal subgroup $Z\triangleleft\hat G$ with $Z\cap\hat H=1$ and an isomorphism $f\colon\hat G/Z\to G$ inducing isomorphisms $\hat H\to H$ and $\hat G/Z\hat H\to Q$.

To define these abstract groups, a key observation is the result by Day-Wade that the groups in their series are generated by the partial conjugations that they contain (Lemma~\ref{lem:generating})

To fix notation, let $\Omega_G^v$ for $v\in\Gamma$ be the family of sets associated to the $\PAut$-like structure of $G$. If $v\in\Delta$, let
\[\Omega_P^v=\{A\in\Omega_G^v\mid A\cap\Delta\neq\emptyset\}\]
and
\[\Omega_H^v=\{L\in\Omega_G^v\mid L\cap\Delta=\emptyset\}\cup\{S^v\}\]
where $S^v=\cup\{A\in\Omega_G^v\mid A\cap\Delta\neq\emptyset\}$. And for $w\in\Gamma\setminus\Delta$, we set $\Omega_H^w=\Omega_G^w$.

Note that the assumption that no vertex $v\in\Delta$ is linked to all the other vertices of $\Delta$ implies $S^v\neq\emptyset$ for any $v\in\Delta$.

Observe that each of the sets $\Omega^v_P$ for $v\in\Delta$ lies inside $\Omega^v_G$ thus the whole family satisfies condition ii) in the definition of $\PAut$-like \ref{def:PAut} above. We are going to see that the same happens for the $\Omega_H^w$'s. We will need the following lemma.

\begin{lemma}\label{lem:condH} With the previous notation, suppose that $A_\Delta\in\mathcal{H}$. If $w\in\Gamma\setminus\Delta$, then there is at most one $S^w\in\Omega_G^w$ that intersects $\Delta$.
\end{lemma}
\begin{proof} By Lemma \ref{DayWade2.2}, if $c^w_S$ preserves $A_\Delta$ we must have either $S\cap\Delta=\emptyset$, or $\Delta\cap(\Gamma\smallsetminus\st_\Gamma(v))\subseteq S$. Since the elements of $\Omega_G^w$ are disjoint, there exists at most one $S \in \Omega_G^w$ satisfying $\Delta\cap(\Gamma\smallsetminus\st_\Gamma(v))\subseteq S$, and the result follows.
\end{proof}

\begin{lemma} The family of sets $\Omega^w_H$ for $w\in\Gamma$ satisfies condition ii) in the definition of $\PAut$-like \ref{def:PAut} above.
\end{lemma}

\begin{proof} Let $u,w\in\Gamma$ be non-adjacent vertices. We need to check that the elements of $\Omega_H^u$ and $\Omega_H^w$ are split into the three possible cases (dominant, shared or subordinate). The assertion is trivial if both $w,u$ lie in $\Gamma\setminus\Delta$ because then $\Omega^w_G=\Omega^w_H$ and $\Omega^u_G=\Omega^u_H$.

Assume that $u,w\in\Delta$. There is only one $S^u\in\Omega_H^u$ and only one $S^w\in\Omega_H^w$ that meet $\Delta$, so they must be dominant. The remaining elements of $\Omega_H^u$ and $\Omega_H^w$ are the shared or subordinate components of $\Omega_G^u$ and $\Omega_G^w$ that do not meet $\Delta$, and clearly they retain their status as shared or  subordinate  in $\Omega_H^u$ and $\Omega_H^w$.

By symmetry it remains to consider the case $w \in \Delta$ and $u \in \Gamma \smallsetminus \Delta$. If  $A \in \Omega_H^u$ is not dominant then, similarly to the above, it is either a shared or a subordinate component from $\Omega_G^u$ that does not meet $\Delta$, and it remains shared or subordinate in $H$, since $\Omega_H^u = \Omega_G^u$. On the other hand, if $B \in \Omega_H^w$ is not dominant, then we have two cases:
\begin{enumerate}
 \item $B$ meets $\Delta$: then it is subordinate in $H$, since the dominant component in $\Omega_H^u$ contains all components that meet $\Delta$; or
 \item $B$ does not meet $\Delta$: then it retains its status as shared or subordinate  in the passage from $G$ to $H$.
\end{enumerate}

So, in any case, any element $L$ in $\Omega_H^u$ or $\Omega_H^w$ is shared, dominant or subordinate, and the proof is complete.
\end{proof}

Consider now the family of sets $\Omega^v_P$ for $v\in\Delta$. As we observed above, they satisfy ii) in Definition \ref{def:PAut} but note that the elements in each set  form a subgraph of $\Gamma$  only, not necessarily of $\Delta$. The next lemma implies that we would get exactly the same outcome if we consider the intersection of each subgraph with $\Delta$.

\begin{lemma}\label{lem:intersection} With the same notation as above, for $v_1,v_2\in\Delta$ consider $\Omega_P^{v_i}=\{A\in\Omega_G^{v_i}\mid  A\cap\Delta \neq \emptyset\}$ for $i=1,2$.
Consider also the sets $\Omega_P^{v_i}(\Delta)=\{A\cap\Delta\mid A\in\Omega_G^{v_i}\}$.
 Then for each $A\in\Omega_P^{v_1}$
\begin{itemize}
\item[i)] $A$ is dominant for $\Omega_P^{v_2}$ if and only if $A\cap\Delta$ is dominant for $\Omega_P^{v_2}(\Delta)$,

\item[ii)]  $A$ is shared for $\Omega_P^{v_2}$ if and only if $A\cap\Delta$ is shared for $\Omega_P^{v_2}(\Delta)$,

\item[iii)]  $A$ is subordinate for $\Omega_P^{v_2}$ if and only if $A\cap\Delta$ is subordinate for $\Omega_P^{v_2}(\Delta)$.
\end{itemize}
\end{lemma}
\begin{proof} i) is trivial because $v_2\in\Delta$. For ii) and iii) note that the rightward implications are also trivial. In fact, this implies the leftward implications too.
\end{proof}

Before defining the abstract groups $\hat H$ and $\hat P$, we need one more remark about $G$. As we will need later to construct a semidirect product $\hat G=\hat P\ltimes H$  we have to define an action of $\hat P$ on $\hat H$, or more precisely, a group homomorphism $\hat P\to\Aut(\hat H)$. To do that we need to understand first the action of the $\Omega_P^v$-elements of $G$ on the normal subgroup $H$.

\begin{lemma} Let $h=c^w_T$ for $T\in\Omega^w_H$ and $w\in\Gamma$, and let $g=c^v_ A$ where $A$ is a union of sets in $\Omega_P^v$ for $v\in\Delta$. Then
$h^g=h$ unless all the following conditions hold: $w\in A$, and either $v\in T$ or $T\in\Omega_G^w$ is shared for $\Omega_G^v$. 

 When the last conditions do hold, there is some $g_1$ such that $h^{g_1}=h$ and $g_1g = gg_1\in H$  so
$h^g=h^{gg_1}$. 
\end{lemma}

\begin{proof}  We check first that the action is trivial if $w\not\in A$. This includes the cases when $v=w$ or $v$ and $w$ are linked in $\Gamma$, in which the action is obviously trivial, so we assume that $v\neq w$ are not linked. Recall that there is only one element $S^w\in\Omega_H^w$ that meets $\Delta$ so $v\in S^w$. Then:

If $v\not\in T$, then $T$ also lies in $\Omega_G^w$ and can be either shared of subordinate for $\Omega_G^v$. If it was shared it would not be one of the components of $A$ (because they all meet $\Delta$ and $T$ does not) so in both cases we get $h^g=h$.

If $v\in T$, then $T=S^w$. Now, we distinguish two cases. Assume first that $w\not\in\Delta$. Then $S^w\in\Omega_G^w$ is the dominant component for $\Omega_G^v$. We have $w\not\in A$ and $A$ is a union of components of $\Omega_G^v$ which can not be dominant (because $w\not\in A$) and can not be shared (because they are different from $S^w$ and they meet $\Delta$). So they are all subordinate, i.e., $h^g=h$.

We are left with the case when $w\in\Delta$. Then $S^w$ is a union of those components in $\Omega^w_G$ that meet $\Delta$. One of them contains $v$ so one of them is the dominant component for $\Omega_G^v$. As $w\not\in A$, $A$ is a union of subordinate and shared components for $\Omega_G^w$. But all those shared components meet $\Delta$ so they end up in $S^w$. This means that $S^w$ contains both the dominant and each of those shared that may be in $A$. Thus $h^g=h$.

This finishes the proof that the action is trivial if $w\not\in A$ so we assume now that $w\in A$. Clearly, if $T\subseteq A$ the action is trivial (use for example Lemma \ref{lem:relators}). Then the only cases when it might not be trivial are when $v\in T$, i.e., $T=S^w$ or when $T$ is shared for $\Omega_G^v$.

For the last paragraph, note first that if $g\in H$ we can take $g_1=1$. And if $g\not\in H$ then there must be some set in $\Omega_P^v$ not in $A$. Let $L$ be the union of those sets and $g_1=c^v_L$. We have $h^{g_1}=h$ because $w\not\in L$ and $gg_1\in H$ because $A\cup L=S^v\in\Omega^v_H$.
\end{proof}

Note that the condition $w\in A$ in the statement implies that $v\neq w$ are not linked in $\Gamma$ and that if $v\in T$ we must have that $T=S^w$ is both dominant and the only element of $\Omega_H^w$ that meets $\Delta$.

\begin{remark} As the elements $c_L^v$ with $L \in \Omega_P^v$ and $v\in\Delta$ together with $H$ generate the group $G$, this lemma implies that $G$ acts trivially on the abelianization $H/H'$.
\end{remark}

We can define now the abstract groups $\hat H$ and $\hat P$ as the groups given by the $\PAut$ presentation as in Definition \ref{def:PAut} for the generating sets $\Omega^w_H$ for $w\in\Gamma$ and $\Omega^v_P$ for $v\in\Delta$. To do that, we denote the standard generating system for $\hat H$ by $X$ so that there is a bijection $s\colon X\to\cup_{w\in\Gamma}\Omega^w_H$ ($s$ from {\sl set}). We will use the superscript $w$ to make clear that the image by $s$ of a given element $a^w\in X$ lies in $\Omega^w_H$ and use the notation $s$ also for union of sets in the same $\Omega^w_H$, so if $x=x^w$ and $x_1=x_1^w$, then $s(xx_1)=s(x)\cup s(x_1)$.  Similarly, we denote the standard generating system for $\hat P$ by $Y$ so that there is a bijection $s:Y\to\cup_{v\in\Delta}\Omega^v_P$ with the same convention about superscripts.
We also denote by $R_{\hat H}^{u,w}$ for $u,w\in\Gamma$ and $R_{\hat P}^{v_1,v_2}$ for $v_1,v_2\in\Delta$ the set of relators as in Definition \ref{def:PAut} but in terms of $X$ and $Y$ resp. (via the bijection $s$). Then we set
\[\hat H=\langle X\mid R_{\hat H}^{u,w},u,w\in\Gamma\rangle,\]
\[\hat P=\langle Y\mid R_{\hat P}^{v_1,v_2},v_1,v_2\in\Delta\rangle.\]

Let $x=x^w\in X$ and $y=y^v\in Y$ be standard generators with $w\in\Gamma$ and $v\in\Delta$. Let $s(y)=A\in\Omega_P^v$, $s(x)=T\in\Omega^w_H$, $g=c^v_A$ and $h=c^w_T$. Let also $z=z^v\in X$ be such that $s(z)=S^v\in\Omega^v_H$. Then we define the action of $\hat{P}$ on $\hat{H}$ by
\begin{equation} \label{def.action.P.H}
x^y\coloneqq \Bigg\{
\begin{aligned}
&x\text{ if }h^g=h\\
&x^z\text{ otherwise.}\\
\end{aligned}
\end{equation}
In other words, we set $x^y=x$ unless all the following conditions hold: $w\in A$, and either $v\in T$ or $T\in\Omega_G^w$ is shared for $\Omega_G^v$.

\begin{remark}  \label{rem:action}
We collect here for further reference some properties of the action above that follow from the definition.
\begin{itemize}
\item[i)] The definition can be extended to allow $y$ be also a product of standard generators in $Y$ all having $s$-image in the same $\Omega_P^v$.  More precisely, we allow $s(y)=A$ for $A$ a union of some of the sets in $\Omega^v_P$ and denote by $y_1=y_1^v$ the element such that for $L=s(y_1)$ we have $A\cap L=\emptyset$ and $A\cup L=\cup\Omega^v_P=S^v\in\Omega^v_H$ (here, $L$ could be empty which means that $y_1=1$). Finally, we denote by $z=z^v\in X$ the element with $s(z)=S^v$. Then we have
for any $x=x^w$, either $x^y=x$ or $x^{y_1}=x$. Moreover, in both cases $x^{yy_1}=x^z$.

\item[ii)] With the same notation as in i), for any $x=x^w$, if $x\neq x^y$ then we must have $w\in A = s(y)$. So we deduce that $x_1^{y_1}=x_1$  for any $x_1=x_1^w$.

\item[iii)] If $x=x^w$ is such that $s(x)=T$ is shared for $\Omega_G^w$ and $\Omega_G^v$, then $[x,zt]=1$ in $\hat H$ (it follows from the $\PAut$-like structure) for $t=t^v\in X$ such that $s(t)=T$. Therefore
\[(x^y)^t=(x^{yy_1})^t=x^{zt}=x.\]
\end{itemize}
\end{remark}

In the next two lemmas we show that \eqref{def.action.P.H} determines a well-defined action of $\hat{P}$ on $\hat{H}$.

\begin{lemma}\label{lem:welldef1} For any $y=y^v\in Y$ standard generator, the map $x\mapsto x^y$ extends to a group isomorphism $\xi_y \colon \hat H\to\hat H$ which induces the identity map on the abelianization $\hat H/\hat H'$.
\end{lemma}
\begin{proof} We have to check that $\xi_y$ preserves the relators in $R_{\hat H}^{u,w}$ for any $u,w\in\Gamma$. Let $[\alpha,\beta]=1$ with $\alpha=\alpha^w$ and $\beta=\beta^u$ be one of these relators where we allow either $\alpha$ or $\beta$ (but not both) to be a product of two standard generators. For the proof it will be also convenient to allow also $y$ to be  a product of standard generators in $Y$ all having $s$-image in the same $\Omega_P^v$, as we did in Remark~\ref{rem:action}. We will use the same notation as there, in particular $A=s(y)$ and there is an element $y_1=y_1^v$ with $L=s(y_1)$ and $z\in X$ with $s(z)=A\cup L$ so that the action of the composition $\xi_y\xi_{y_1}$ is conjugation by $z$.

Observe first that if 
\begin{equation}\label{eq:cond}
\text{either }\alpha^y=\alpha,\beta^y=\beta\text{ or }\alpha^{y_1}=\alpha,\beta^{y_1}=\beta
\end{equation}
then the relation is preserved. In the first case this is obvious because
\[[\alpha^y,\beta^y]=[\alpha,\beta]=1\]
and in the second we have
\[[\alpha^y,\beta^y]=[\alpha^{yy_1},\beta^{yy_1}]=[\alpha^{z},\beta^{z}]=[\alpha,\beta]^z=1.\]

We have (\ref{eq:cond}) if $\alpha^y=\alpha^{y_1}=\alpha$. To see it, note that by Remark \ref{rem:action} i) and ii), either $\beta^y=\beta$ or $\beta^{y_1}=\beta$. 
This is the case if $v=w$ or $v$ and $w$ are linked or, more generally, if $w\not\in A\cup L$. So we may assume now that $w\in A$. In fact, by symmetry we can also assume that $u\in A\cup L$. If both $u,w$ are in $L$, then we have (\ref{eq:cond}) again so from now on we assume $w\in A$, $u\in L$, the opposite case being analogous. Observe that as $A$ and $L$ are unions of connected components of $\Gamma\setminus\st_\Gamma(v)$, this implies that $w\neq u$ are not linked in $\Gamma$. We may also assume that $\alpha^y\neq \alpha$ and $\beta^{y_1}\neq\beta$ and by definition this implies that either $v\in s(\alpha)$ or $s(\alpha)$ contains a shared component for $\Omega_G^w$ and $\Omega_G^v$, same considerations apply to $\beta$.

Consider the set $\Omega_H^w$. We have
\[\Omega_H^w=\{S^w,W_1,\ldots,W_k\}\]
where $S^w$ is the only one that meets $\Delta$, moreover the $W_i$ lie also in $\Omega_G^w$. With respect to $\Omega_G^v$, each $W_i$ can be either subordinate, which implies $W_i\subseteq A$ (recall that we are assuming $w\in A$) or shared, which means that $W_i$ is one of the elements of $\Omega_G^v$ which do not meet $\Delta$. In both cases we deduce that $u\not\in W_i$ because $u\in L$. So we see that $u\in S^w$, i.e., $S^w$ is also the dominant component for $\Omega_H^u$. By the same argument, $S^u\in\Omega_H^u$ is the dominant component for $\Omega_H^w$.

The last paragraph implies that in the relator $[\alpha,\beta]$ we can not have at the same time $v\in s(\alpha)$ and $v\in s(\beta)$ because then $S^w\subseteq s(\alpha)$ and $S^u\subseteq s(\beta)$ and since these are the dominant components for $u,w$ this is not a valid relator in the $\PAut$-like structure of $\hat H$. So by symmetry we only have to consider the following two possibilities:

\textbf{Case 1:} $s(\alpha)=W$ is shared for $w$ and $v$ and $s(\beta)=U$ is shared for $u$ and $v$, moreover we must have $W\neq U$. In this case $W$ does not meet $\Delta$ so there is $t=t^v\in X$ with $s(t)=W$ with $(\alpha^y)^t=\alpha$ (see Remark \ref{rem:action} iii)). As $W\neq U$, we also have $[\beta,t]=1$ in $\hat H$ so
\[[\alpha^y,\beta^y]^t=[\alpha,\beta]=1.\]

\textbf{Case 2:} $s(\alpha)=S^w\cup T$ and $s(\beta)=T$ where $T$ is shared for $w,u$ and also shared for $u,v$ (and does not meet $\Delta$). Again, let $t=t^v\in X$ with $s(t)=T$. Then $[\alpha,t]=1$ is a relator of $\hat H$. We also have $\alpha^{y_1}=\alpha$ (Remark \ref{rem:action} ii)) and $\beta^y=\beta$. And on the other hand, Remark \ref{rem:action} iii) implies $(\beta^{y_1})^t=\beta$. So we get
\[[\alpha^y,\beta^y]=[(\alpha^{yy_1})^t,(\beta^{yy_1})^t]=[\alpha,\beta]^{zt}=1.\qedhere \]
\end{proof}

\begin{lemma}\label{lem:welldef2} The map $y\mapsto\xi_y$ extends to a homomorphism $\xi\colon \hat P\to\Aut(\hat H)$.
\end{lemma}
\begin{proof} We have to check that $\xi$ preserves the relators of $\hat P$. Let $v_1,v_2\in\Delta$ and consider $a=a^{v_1}$, $b=b^{v_2}$ such that $[a,b]=1$ in $\hat P$ where $a$ or $b$ are either standard generators or one of them  (but not both) can be a product of two standard generators. Let $h=h^w\in X$ be a standard generator, we claim that \begin{equation} \label{action.commutes} (h^a)^b=(h^b)^a. \end{equation}
Let $A=s(a)$ and $B=s(b)$. By Lemma \ref{lem:relators} one of the following three options hold: either i) $A\cap B=\emptyset$, $v_1\not\in B$ and $v_2\not\in A$; or ii) $A\subseteq B$ and $v_1\in B$; or iii) $B\subseteq A$ and $v_2\in A$. A first observation is that \eqref{action.commutes} holds trivially when $v_1=v_2$ or if $v_1$ and $v_2$ are linked so we assume this is not the case.

Assume first that 
 $h^a=h$. If also $h^b=h$, the claim is obvious. So we assume $h^b\neq h$. By Remark \ref{rem:action} this implies that there is a $b_1\in\hat P$ and $z=z^{v_2}\in\hat H$ such that $h^b=h^{bb_1}=h^z$. We also deduce $w\in B$. As $z\in\hat H$, again by Remark~\ref{rem:action} we see that if $v_2\not\in A$, then $z^a=z$.
This implies the claim because we have:
\[(h^a)^b=h^b=h^z,\]
\[(h^b)^a=(h^z)^a.\]
So in cases i) or ii) above we are done. We assume that we have iii). i.e., that  $B\subseteq A$ and $v_2\in A$. Then $w\in B$ implies $w\in A$. Moreover, $w$ and $v_1$ can not be linked 
because $w \in A \subset \Gamma \smallsetminus \st(v_1)$. The fact that $h^a=h$ implies that $s(h)$ can not be $S^w$, i.e., the only element of $\Omega^w_H$ that meets $\Delta$. So taking into account that $h^b\neq h$ then Remark \ref{rem:action} implies that $s(h)=T\in\Omega^w_H$ must be shared in $\Omega_G^{v_2}$. As $h^a=h$, $T$ can not be shared in $\Omega_G^{v_1}$. 
Obviously, $v_1\not\in T$ so $T$ (which also lies in $\Omega_G^{w}$) must be subordinate for $\Omega_G^{v_1}$, i.e., $T\subseteq A$. Let $t=t^{v_2}\in X$ with $s(t)=T$, then by Remark \ref{rem:action} iii) we have $(h^b)^t=(h^z)^t=h^{zt}=h$.  Moreover, we have $t^a=t$. Therefore

\[((h^a)^b)^t=(h^b)^t=h^{zt}=h,\]
\[((h^b)^a)^t=((h^{z})^a)^t=t^{-1}(h^z)^at=(h^{zt})^a=h^a=h\]
so we get the claim.

Next, we assume that $h^a\neq h$ and $h^b\neq h$. Then we must have $w\in A\cap B$ so we can not have i) above, i.e. either ii) or iii) must hold true. By symmetry we can assume for example iii), so $B\subseteq A$ and $v_2\in A$. Moreover, $h^a=h^{z_a}$ and $h^b=h^{z_b}$.
We also have $z_a^b=z_a$ because $v_1\not\in B$. And $z_b^a=z_b^{z_a}$ because $v_2\in A$ and $s(z_b)=S^{v_2}$ is the only element in $\Omega_H^{v_2}$ that meets $\Delta$ so it contains $v_1$. Therefore 
\[(h^a)^b=(h^{z_a})^b=(z_a^{-1})^bh^bz_a^b=z_a^{-1}h^{z_b}z_a=h^{z_bz_a},\]
\[(h^b)^a=(h^{z_b})^a=(z_b^{-1})^ah^az_b^a=(z_b^{-1})^{z_a}h^{z_a}(z_b^{z_a})=(h^{z_b})^{z_a}\]
so we are done.
\end{proof}

Now, we can define a semidirect product $\hat G=\hat P\ltimes_\xi\hat H$. In $\hat G$ we are going to define a subgroup $Z$ as follows. For any $v\in\Delta$, there is a $z = z^v\in\hat H$ such that $s(z)=S^v\in\Omega_H^v$. Let $y_1 = y^v_1,\ldots,y_t = y^v_t$ be all the elements in $Y$ which are the preimage under $s$ of $\Omega^v_P$ and let $Z$ be the group generated by all the elements of the form
\[(\prod_iy_i)z^{-1}=z^{-1}(\prod_iy_i).\]
for $v\in\Delta$.

\begin{lemma}\label{lem:Z} The subgroup $Z$ is normal in $\hat G$ and $Z\cap\hat H=1$.
\end{lemma}
\begin{proof} Observe first that Remark \ref{rem:action} i) implies that the generators of $Z$ commute with the elements of $X$ so $Z$ is centralized by $\hat H$ and $Z\cap\hat H$ lies in $Z(\hat H)$, the center of $\hat H$.

 Take $v,\bar v\in\Delta$, $\bar y=\bar y^{\bar v} \in Y$
and as before, $z=z^v\in\hat H$ such that $s(z)=S^v\in\Omega_H^v$ and let  $y_1,\ldots,y_t$ be all the elements in $Y$ which are the preimage under $s$ of $\Omega^v_P$  so that $\nu=(\prod _iy_i)z^{-1}\in Z$ is one of the generators.

In order to show that $Z$ is normal, it remains to check that $\nu^{\bar y}\in Z$, for which we distinguish two cases. First, if $z^{\bar y}=z$, then the fact that $s(z)=S^v$ implies $v\not\in s(\bar y)$. This means that $s(\bar y)$ can be either subordinate or shared for $v$ in the $\PAut$-like structure of $\hat P$ and therefore $\bar y$ commutes with the product $\prod_iy_i$. In other words we have
$(\prod_iy_i)^{\bar y}=(\prod_iy_i)$ so $\nu^{\bar y}=\nu \in Z$.

Next, we consider the case in which $z^{\bar y} \neq z$. If we let  $\bar y_1,\ldots,\bar y_s$ be all the elements in $Y$ which are the preimage under $s$ of $\Omega^{\bar v}_P$, we have
 $z^{\bar y}=z^{\prod_i\bar y_i}$ and also $(\prod_iy_i)^{\bar y}=(\prod_iy_i)^{\prod_i\bar y_i}$ thus
\[\nu^{\bar y}=\nu^{\prod_i\bar y_i}=\nu^{\bar z^{-1}\prod_i\bar y_i} = \nu^{\bar \nu} \in Z,\]
where  $\bar z\in\hat H$ is such that $s(z^v)=S^v\in\Omega_H^v$ and $\bar\nu= \bar z^{-1}\prod_i\bar y_i \in Z$. So $Z$ is a normal subgroup of $\hat{G}$.

Now, observe that with the same notation as in the previous paragraph, the fact that $\hat H$ is centralized by $Z$ implies that
\[\nu\bar\nu=(\prod _iy_i)z^{-1}(\prod_i\bar y_i)\bar z^{-1}=(\prod _iy_i\prod_i\bar y_i)(z^{-1}\bar{z}^{-1}).\]
Iterating, we see that elements in $Z$ can be written as pairs $ab$ where $a$ is in the subgroup of $\hat P$ generated by the $\prod_iy_i$'s and $b$ in the subgroup $L$ of $\hat H$ generated by the $z$'s. Therefore elements in $Z\cap\hat H$ lie in $L\cap Z(\hat H)\leq Z(L)$. We claim that $L$ is isomorphic to the right angled Artin group $A_\Delta$. Using the 
 hypothesis that no vertex of $\Delta$ is linked to all the other vertices of $\Delta$, we will deduce that $1=Z(A_\Delta)=Z(L)$ so the claim implies the result.

The key observation to show this is that given $u,v\in\Delta$, the elements $z^v$ and $z^u$ commute if and only if $u$ and $v$ are linked, because both are the dominant components in the $\PAut$-like structure of $\hat H$. Moreover, there is no mixed relator $[ab,c]$ with $c=z^u$ and $a$ or $b$ equal to $z^v$.

Consequently we can define a map  $X\to A_\Delta$ given by $z=z^v\mapsto v$ if $v\in\Delta$ and $s(z)=S^v\in\Omega_H^v$ and $x\mapsto 1$ otherwise.
This map extends to a group homomorphism $\hat H\to A_\Delta$. Conversely, we can define a map $\Delta\to X$ given by $v\mapsto z^v$ that extends to a group homomorphism $A_\Delta\to\hat H$ whose image is precisely $L$. Composing both maps yields the identity in $A_\Delta$ so we deduce $L\cong A_\Delta$.
\end{proof}

\begin{remark}\label{rem:Z} It follows from the previous proof that we have an isomorphism of short exact sequences
\[\begin{tikzcd}
1\arrow[r]&Z\hat H/\hat H\arrow[r] \arrow[d] &\hat G/\hat H\arrow[r] \arrow[d]  &\hat G/Z\hat H\arrow[r] \arrow[d]  &1\\
1\arrow[r]&\Inn(A_\Delta)\arrow[r]&P\arrow[r] &Q\arrow[r] & 1.
\end{tikzcd}\]
\end{remark}

\begin{theorem} There is an isomorphism $\hat G/Z\to G$ that induces isomorphisms $\hat H\to H$ and $\hat Q\to\hat G/Z\hat H\to Q$. Moreover $\hat Q\cong\hat G/Z\cong P/\Inn(A_\Gamma)$ where $P$ is the subgroup of $G$ generated by the partial conjugations based at $\Omega_P^v,v\in\Delta$. 
\end{theorem}
\begin{proof} Recall that we are assuming that $G$ is $\PAut$-like. This yields a presentation where the standard generators are of the form $c^w_W$ for $W\in\Omega^w_G$ and $w\in\Gamma$. We can modify this presentation adding:

\begin{itemize}
\item[i)] generators of the form $c^v_{S^v}$ for each $v\in\Delta$ where recall that $S^v\in\Omega^v_H$ is the union of the components of $\Omega^v_H$ that meet $\Delta$ and 
\item[ii)] relators $c^v_{S^v}=\prod\{c^v_A\mid A\in\Omega^v_G,A\cap\Delta\neq\emptyset\}$.
\end{itemize}
Consider the map
$a\mapsto c^u_{s(a)}$ where $a\in X\cup Y$ is a standard generator of $\hat G$. As all the relators of $\hat P$ and $\hat H$ hold in $G$ and the action of $\hat P$ on $\hat H$ was defined using  identities of $G$, we deduce that this maps preserves all the relators of $\hat G$ so it induces a well defined group homomorphism $\hat G\to G$. This map obviously factors through the projection $\hat G\to \hat G/Z$ so we have a group homomorphism $\hat G/Z\to G$. 

For the converse, we will check that the map
\[
\iota \colon c^v_A\mapsto
x^v\text{ with }s(x^v)=A\]
defined on the standard generators of $G$ can be extended to a group homomorphism, i.e., we have to check that the relators are preserved. Let
\[Y_G=\{c^v_A\mid v\in\Delta,A\cap\Delta\neq\emptyset\},\]
\[X_G=\{c^v_A\mid v\in\Delta,A\cap\Delta=\emptyset\}\cup\{c^w_A\mid w\in\Gamma\setminus\Delta\}.\]
Then $\iota$ send elements of $Y_G$ (resp. $X_G$) to $Y$ (resp. $X$). And relators of the form $[a,b]$ where $a$ and $b$ are single standard generators are sent to relators of $\hat P$ if both $a,b\in Y_G$, to relators of $\hat H$ if both $a,b\in X_G$ and to relators describing the action of $\hat P$ on $\hat H$ if $a\in Y_G$ and $b\in X_G$ or the other way around. 
 Relators of the form $[a,b]$ where both are single standard generators lie either in $\hat P$ or in $\hat H$ or are given by the action. That is to say, in all these cases the relator is preserved. The same thing happens for relators $[ab,c]$ where either all of $a,b,c$ lie in $Y_G$ or all them lie in $X_G$.

So we only have to check what happens with relators $[ab,c]$ where elements of $Y_G$ and $X_G$ are mixed. Put $c=c^v_T$. The $\PAut$-like hypothesis implies that one of $a,b$, say $a$, must be dominant for $v$, i.e. $a=c^w_A$ with $v\in A$ and $b$ must be shared, i.e., $b=c^w_T$. We claim first that if $c\in Y_G$, then also $a,b$ must be in $Y_G$.  As $c\in Y_G$, $v\in\Delta$ and $T\cap\Delta\neq\emptyset$. Then we see that $A$ and $T$ are two different components of $\Omega^w_G$ that meet $\Delta$. This is impossible if   $w\not\in\Delta$,  because $\Delta$ is $G$-invariant. So $w\in\Delta$ and $a,b\in Y_G$.

Therefore we may assume that $c\in X_G$. Assume that $b$ lies in $Y_G$ so $w\in\Delta$. Then $T\cap\Delta\neq\emptyset$ so we deduce that $v\not\in\Delta$. But again this is impossible because it would imply that at least $T$ and the component of $\Omega_G^v$ that contains $w$ meet $\Delta$. Therefore either both $a,b$ lie in $X_G$ or $a\in Y_G$. In the first case we are done.  And in the second, the relator $[ab,c]$ can be rewritten as  $(c^a)^bc^{-1}$ and applying $\iota$ we get a relator of $\hat G$ because of Remark \ref{rem:action} iii).
\end{proof}

\begin{corollary} Let $G=\PAut(A_\Gamma,\mathcal{G},\mathcal{H}^t)$ be a group that is obtained in a Day-Wade series starting from $\PAut(A_{\Gamma_1})$, where $\Gamma$ is a subgraph of $\Gamma_1$. Then $G$ is $\PAut$-like and $\overline{G}=G/\Inn(A_\Gamma)$ is $\POut$-like.
\end{corollary}

\section{Some exact sequences of Lie algebras}

Here we explain why the group decompositions also work on the level of their graded Lie algebras. The first step is to show that the map $\PAut(A_{\Gamma}) \to \POut(A_\Gamma)$ induces a short exact sequence
\[ 0\to \gr(\Inn(A_\Gamma))\to\gr(\PAut(A_\Gamma))\to\gr(\POut(A_{\Gamma}))\to 0\]
of graded Lie algebras. Notice that $\Inn(A_\Gamma)$ is isomorphic to the RAAG defined on the subgraph $\Gamma_0 \subseteq \Gamma$ determined by the vertices which are not connected to every other vertex, so $\gr(\Inn(A_\Gamma))$ is actually a right-angled Artin Lie algebra. We will work with the presentations for $\gr(\PAut(A_\Gamma))$ and $\gr(\POut(A_\Gamma))$ given in Theorem~\ref{teoB}, in which generators are in bijection with the set of partial conjugations $\cup_v \Omega_G^v$, for $G = \PAut(A_\Gamma)$.

For $\al = \gr(A_{\Gamma_0})$, define $\varphi \colon \al \to \gr(\PAut(A_\Gamma))$ to be the homomorphism sending
$v$ to the sum $\sum c_L^v$ running through all $L \in \Omega_G^v$. Here we are thinking of $v \in \Gamma_0 \subseteq \Gamma$ as a generator of $\al$. Notice that $\im(\varphi)$ is an ideal of $\gr(\PAut(A_\Gamma))$. Indeed, for any generator $v \in \al$ and $c_M^w \in \gr(\PAut(A_\Gamma))$, either $[\varphi(v), c_M^w]=0$, or $[\varphi(v), c_M^w]= [\varphi(v), \varphi(w)]$, so in any case we get an element of $\im(\varphi)$.

It is clear from the presentations of  $\gr(\PAut(A_\Gamma))$ and $\gr(\POut(A_\Gamma))$ that $\im(\varphi)$ is equal to the kernel of $\gr(\PAut(A_\Gamma))\to\gr(\POut(A_{\Gamma}))$, so it remains to show that $\varphi$ is injective. For that we adapt the proof of \cite[Proposition~3.2]{ConnerGoetz}.
The idea is to define a homomorphism $\psi \colon \gr(\PAut(A_\Gamma)) \to \Der(\al)$ that makes the following diagram commutative:
\begin{equation}  \label{diagram.psi.ad}
\xymatrix{ \gr(\PAut(A_\Gamma)) \ar[r]^{\psi} & \Der(\al) \\
 \al   \ar[u]^{\varphi} \ar[ur]_{\ad} & }
\end{equation}
Notice that $\al$ has trivial center, so $\ad: \al \to \Der(\al)$ is injective. Therefore, if we can define such a $\psi$, we get from the commutativity of the diagram that $\varphi$ is injective.

We define $\psi$ as
\[
\psi(c_L^w) ( v ) =
     \begin{cases}
       [w,v]  &\quad\text{if } v \in L,\\
       0  &\quad\text{otherwise,} \\
     \end{cases}
\]
and extend it by the Leibniz rule. There are many things to verify. We will use the following standard result.

\begin{lemma} \label{defder}
 Let $F$ be a free Lie algebra on the set $X$ and let $R$ be the ideal of $F$ generated by some elements $r_1, \ldots, r_n \in F$.
 Then:
 \begin{enumerate}
  \item Any function $\delta \colon X \to F$ extends to a derivation $\Delta: F \to F$;
  \item If $\Delta(r_i) \in R$ for all $i$, then $\Delta(R) \subseteq R$ and the induced map $\bar{\Delta} \colon F/R \to F/R$ is a derivation.
 \end{enumerate}
\end{lemma}

\begin{lemma}
$\psi(c_L^w)$ is a well defined derivation of $\al$ for all $L$ and $w$.
\end{lemma}

\begin{proof}
Clearly we only need to work with $w \in \Gamma_0$.
By Lemma \ref{defder} it is enough to show that if $v_1, v_2 \in \Gamma_0$ are adjacent, then the element
\begin{equation} \label{eqverif}
[\psi(c_L^w)v_1, v_2] +  [v_1,\psi(c_L^w) v_2]
 \end{equation}
is a consequence of the defining relators for $\al$.

If  $v_1, v_2 \notin L$, then \eqref{eqverif} reduces to $0$. Similarly, if $v_1, v_2 \in L$, then:
\[[\psi(c_L^w)v_1, v_2] +  [v_1,\psi(c_L^w) v_2] = [[w,v_1],v_2]+[v_1,[w,v_2]] = [w,[v_1,v_2]]\]
by the Jacobi identity. Since $[v_1,v_2]$ is a defining relator of $\al$, we are done in this case.

Finally, suppose that $v_1 \in L$ but $v_2 \notin L$. In this case, since $L$ is a connected component of $\Gamma \smallsetminus \st(w)$, we have $v_2 \in \st(w)$. Then
\[[\psi(c_L^w)v_1, v_2] +  [v_1,\psi(c_L^w) v_2] = [[w,v_1],v_2] = [[w,v_2],v_1]+[w,[v_1,v_2]].\]
Since $[w,v_2]$ and $[v_1,v_2]$ are defining relators, we are done. The case where $v_1 \notin L$ and
$v_2 \in L$ is completely analogous.
\end{proof}

\begin{lemma}
 $\psi \colon  \gr(\PAut(A_\Gamma)) \to \Der(\al)$ is a homomorphism of Lie algebras.
\end{lemma}

  \begin{proof}
  We need to show that the relations of $\gr(\PAut(A_\Gamma))$ are respected. For any
  vertex $w \in \Gamma$ and $L \subseteq \Gamma$ a subgraph, let
  \[
\delta_L^w =
     \begin{cases}
       1  &\quad\text{if } w \in L,\\
       0  &\quad\text{otherwise.} \\
     \end{cases}
\]
Then for any generator $c_L^w \in \gr(\PAut(A_\Gamma))$ and $v \in \Gamma_0$, we have \[\psi ( c_L^w) (v) = \delta_L^v [w, v].\]
 For any generators $c_L^{w_1}, c_M^{w_2} \in \gr(\PAut(A_\Gamma))$ and $v \in \Gamma_0$, using the Jacobi identity and the Leibniz rule for derivations, we have:
 \begin{equation} \label{a1}
  [\psi(c_L^{w_1}), \psi(c_M^{w_2})] (v) = (\delta_M^v \delta_L^{w_2} + \delta_L^v \delta_M^{w_1} - \delta_M^v \delta_L^v)[[w_1,w_2],v].
 \end{equation}

 Now suppose that $[c_L^{w_1}, c_M^{w_2}]=0$ is a defining relation in $\gr(\PAut(A_\Gamma))$. So we are in one of the following cases:
 \begin{enumerate}
  \item[(i)] $w_1 \in \st(w_2)$;
  \item[(ii)] $(L \cup \{w_1\}) \cap (M \cup \{w_2\}) = \emptyset$;
  \item[(iii)] $L \cup \{w_1\} \subseteq M$;
  \item[(iv)] $M \cup \{w_2\} \subseteq L$.
 \end{enumerate}
 It is clear that \eqref{a1} is $0$ in case (i). In case (ii) we have $\delta_L^{w_2} = \delta_M^{w_1} = \delta_L^v \cdot \delta_M^v = 0$, thus clearly \eqref{a1} reduces to $0$ again. In case
 (iii) we have $\delta_L^{w_2}=0$ and $\delta_M^{w_1}=1$, so the coefficient on the right hand side of \eqref{a1} is
 $\delta_L^v (1- \delta_M^v)$, which is zero whenever $v \in M$ or $v \in \Gamma \smallsetminus M \subset \Gamma \smallsetminus L$. Case (iv) is completely analogous to case (iii).

 Now we consider a relation of type $[c_L^{w_1}, c_L^{w_2}+c_M^{w_2}]$, with $w_1 \in M$. We need to show that $A =[\psi(c_L^{w_1}), \psi(c_M^{w_2})](v)$
 and $B = -[\psi(c_L^{w_1}), \psi(c_L^{w_2})](v)$ coincide for all $v \in \Gamma_0$. First, since $L$ is a shared component, we have $\delta_L^{w_1} = \delta_L^{w_2} = 0$, so
 \[ B = -(\delta_L^v \delta_L^{w_2} + \delta_L^v \delta_L^{w_1} - \delta_L^v \delta_L^v)[[w_1,w_2],v] = -(-\delta_L^v \delta_L^v )[[w_1,w_2],v] = \delta_L^v[[w_1,w_2],v].\]
  On the other hand, since $L$ and $M$ are disjoint, $\delta_L^v \cdot \delta_M^v = 0$. So
\[ A = (\delta_M^v \delta_L^{w_2} + \delta_L^v \delta_M^{w_1} - \delta_M^v \delta_L^v)[[w_1,w_2],v] =\delta_L^v[[w_1,w_2],v],\]
as $\delta_M^{w_1} = 1$. Thus $A$ = $B$.
 \end{proof}

\begin{prop}  \label{grInnPAut}
For any finite graph $\Gamma$, there is an exact sequence
\[ 0\to \gr(\Inn(A_\Gamma))\to\gr(\PAut(A_\Gamma))\to\gr(\POut(A_{\Gamma}))\to 0\]
induced by the group homomorphism $\PAut(A_\Gamma) \to \POut(A_\Gamma)$.
\end{prop}

\begin{proof}
It remains to show that the diagram \eqref{diagram.psi.ad} commutes. For $w, v \in \Gamma_0$ we have:
\[ (\psi \circ \varphi (w))(v) = \sum \psi( c_L^w ) (v),\]
where the sum runs over the connected components $L \subseteq \Gamma \smallsetminus \st(w)$. By definition $\psi( c_L^w ) (v)$ can be non-zero
only if $v \in L$. If $w$ and $v$ are adjacent, then $(\psi \circ \varphi (w))(v) = 0 = \ad(w)(v)$. Otherwise, $v \in L_1$ for some
connected component $L_1 \subseteq \Gamma \smallsetminus \st(w)$, so
\[ (\psi \circ \varphi (w))(v) = \psi(c_{L_1}^w) (v) = [w,v] = \ad(w)(v).\]
So $\psi \circ \varphi(w) = \ad(w)$ for all $w$, that is, the diagram is commutes.
 \end{proof}

 Next, we generalize this to $\PAut$-like groups.

\begin{prop}\label{prop:centralextension}
Let $G\leq \PAut(A_{\Gamma})$ be a $\PAut$-like group and $I=\Inn(A_\Gamma)$ so that there is a short exact sequence of groups
\[1\to I\to G\to G/I\to 1.\]
Then there is an induced short exact sequence of graded Lie algebras
\[0\to\gr(I)\to\gr(G)\to\gr(G/I)\to 0.\]
\end{prop}
\begin{proof}  By Proposition~\ref{grInnPAut}, the sequence
\begin{equation}  \label{secgr} 0\to\gr(I)\to\gr(\PAut(A_\Gamma))\to\gr(\POut(A_\Gamma))\to 0 \end{equation}
is exact. Now, for a $\PAut$-like group $G \leq \PAut(A_\Gamma)$, the inclusions
\[I\leq G\leq\PAut(A_\Gamma)\]
induce maps
\[\gr(I)\to\gr(G)\to\gr(\PAut(A_\Gamma))\]
such that the composition  $\gr(I)\to\gr(\PAut(A_\Gamma))$ is the first map in the short exact sequence \eqref{secgr},
so it is injective. This implies that the map $\gr(I)\to\gr(G)$ is also injective which implies the result.
\end{proof}

Using the notation of Section \ref{sec:presentation}, let $G=\PAut(A_\Gamma,\mathcal{G},\mathcal{H}^t)$, $A_\Delta\in\mathcal{G}$, $H=\PAut(A_\Gamma,\mathcal{G},(\mathcal{H}\cup\{A_\Delta\})^t)$ and $Q=\POut(A_\Delta,\mathcal{G}_\Delta,\mathcal{H}_\Delta)$ fitting in Day-Wade short exact sequence
\[1\to H\to G\to Q\to 1.\]
We want to check that there is a similar short exact sequence for the associated descending central Lie algebras.
Let
$\hat G=\hat P\ltimes\hat H$ with $H\cong\hat H$ and $Z\triangleleft\hat G$ with $Z\cap\hat H=1$, $Q\cong\hat G/Z\hat H$ and $\hat G/Z\cong G$. We have an exact commutative diagram:
\[
\begin{tikzcd}
&&Z\arrow[r,equal]\arrow[d,rightarrowtail] &Z\arrow[d,rightarrowtail]\\
& H\arrow[r,rightarrowtail]\arrow[d,equal]&\hat G\arrow[r,twoheadrightarrow]\arrow[d,twoheadrightarrow] &\hat P\arrow[d,twoheadrightarrow]\\
&H\arrow[r,rightarrowtail]&G\arrow[r,twoheadrightarrow]&Q\\
\end{tikzcd}
\]

By Lemma \ref{lem:welldef1}, the group $\hat P$ acts trivially on the abelianization $\hat H/\hat H'$ of $\hat H$. From \cite[Theorem~3.1]{FalkRandell} we deduce that $\gr(\hat G)=\gr(\hat P)\ltimes\gr(\hat H)$. Using this and Proposition \ref{prop:centralextension} we see that the associated descending central Lie algebras fit in an exact commutative diagram
\[
\begin{tikzcd}
&&\gr(Z)\arrow[r,equal]\arrow[d,"\tau"]  &\gr(Z)\arrow[d,rightarrowtail,"\sigma"]\\
& \gr(H)\arrow[r,rightarrowtail,"\hat\iota"]\arrow[d,equal]&\gr(\hat G)\arrow[r,twoheadrightarrow,"\hat\pi"]\arrow[d,twoheadrightarrow, "\nu"] &\gr(\hat P)\arrow[d,twoheadrightarrow]\\
&\gr(H)\arrow[r,"\iota"]&\gr(G)\arrow[r,twoheadrightarrow,"\pi"]&\gr(Q)\\
\end{tikzcd}
\]

Note first that the injectivity of the map $\sigma$ implies that $\im\tau\cap\ker\hat\pi=0$. As $\im\tau=\ker\nu$, we deduce that $\ker\nu\cap\ker\hat\pi=0$. This and the injectivity of $\hat\iota$ imply that also $\iota$ is injective so we get the desired short exact sequence of Lie algebras.
\begin{equation}\label{eq:seq}
0\to\gr(H)\to\gr(G)\to\gr(Q)\to 0.\end{equation}

Using this we deduce

 \begin{theorem}\label{polyRAAGLieSeries} Let $\mathcal{G}$ and $\mathcal{H}$ be sets of special subgroups of $A_\Gamma$. There is a series of graded subideals
 \[1=\gr(N_0)\trianglelefteq\gr(N_1)\trianglelefteq\cdots\trianglelefteq\gr(N_t)=\gr(\PAut(A_\Gamma,\mathcal{G},\mathcal{H}^t))\]
such that, if  the graph $\Gamma$ has property (*), then each quotient $\gr(N_i)/\gr(N_{i-1})$ is a RAAG Lie algebra.
 \end{theorem}
\begin{proof} Use  Theorem~\ref{teoC} together with an iteration of the construction of the short exact sequence (\ref{eq:seq}) and  the semidirect product decomposition of Lemma \ref{FouRab} for the Fouxe-Rabinovitch case. Note that in all cases the subideals are generated by sums of generators (of degree $1$) of $\gr(\PAut(A_\Gamma,\mathcal{G},\mathcal{H}^t))$, so that the sequence is indeed of graded Lie algebras.
\end{proof}

\begin{proof}[Proof of Theorem~\ref{teoA}]
If the graph $\Gamma$ satisfies (*), then by Theorem~\ref{polyRAAGLieSeries} the Lie algebra $\gr(\PAut(A_{\Gamma}))$  admits a graded series of subideals whose factors are the graded Lie algebras associated to (finitely generated) RAAGs. These are Koszul by \cite{Froberg}, thus $\gr(\PAut(A_{\Gamma}))$ is Koszul by Proposition~\ref{polyRAAGLie}. The converse is Proposition~\ref{obstruction}.
\end{proof}

\section*{Acknowledgements}
Part of this work was carried out during the XXVII Brazilian Algebra Meeting, which provided us the opportunity to discuss it in person. We thank the organizers for their invitation and support.
Moreover, the first named author was supported by the Spanish Government PID2021-126254NB-I00 and Departamento de Ciencia, Universidad y Sociedad del 
Conocimiento del Gobierno de Arag{\'o}n (grant code: E22-23R: ``{\'A}lgebra y Geometr{\'i}a'') and the second named author was supported by grants 2015/22064-6 and 2016/24778-9, from São Paulo Research Foundation (FAPESP) and by grant FAPEMIG [APQ-02750-2].


\end{document}